\documentclass[]{article}
\usepackage[a4paper,margin=1in,top=1in, bottom=1in, left=1in, right=1in, headheight=50pt, headsep=15pt]{geometry}

\usepackage{fancyhdr}

\usepackage{graphicx}

\usepackage{ mathrsfs }
\usepackage{amsmath,amssymb,amsfonts,amsthm}
\usepackage{amsrefs}
\usepackage{mathtools}
\usepackage{mathtools}
\usepackage{thmtools}

\usepackage{amsmath}
\usepackage{pgfplots}
\usepackage[shortlabels]{enumitem}

\newtheorem{thm}{Theorem}[section]

\newtheorem{prop}[thm]{Proposition}

\newtheorem{lem}[thm]{Lemma}
\newtheorem{lemma}[thm]{Lemma}

\newtheorem{remark}[thm]{Remark}

\newtheorem{letterthm}{Theorem}

\renewcommand{\Re}{\operatorname{Re}}

\newcommand{\mathcalSsumdef}{\mathfrak{S}}

\newcommand{\sumk}{\sum_{\substack{\kestermann = 1 \\ \gcd(\kestermann,\ell) = 1}}^\ell}

\allowdisplaybreaks

\title{Convolution identities for complex-indexed divisor functions and modular graph functions}

\author{Ksenia Fedosova and Kim Klinger-Logan}

\usepackage{color}
\newcommand{\kf}[1]{{\color{magenta}{#1}}}

\begin{document}

\date{}

\maketitle

\begin{abstract}
We find exact identities for sums of the form 
\begin{equation*}\label{eq:convsumabs}
  \sum_{\stackrel{n_1+n_2 = n}{n_1 \in \mathbb{Z} \setminus \{ 0, n \} }}  \mathcal{Q}(n_1,n_2)  \sigma_{-r_1}(n_1) \sigma_{-r_2}(n_2),
  \end{equation*}
  where $n\in\mathbb{N}$, $r_1,r_2\in\mathbb{C}$, $\mathcal{Q}$ is a combination of hypergeometric functions, and $\sigma_{a}(x)$ denotes the divisor function.  Specifically,  
  we find that they can be expressed in terms of Fourier coefficients of Hecke cusp forms weighted by their $L$-values.   This result expands upon previous work with Radchenko in which such identities were found for divisor functions with  even integer index \cite{FKLR} and encompasses  results of Jacobi \cite{motohashi1994binary} and Diamantis and O'Sullivan in \cite{diamantis2010kernels, o2023identities} for divisor functions with odd integer index. The proof of our result expresses these sums in terms of Estermann zeta functions and uses trace formulae. In addition, we use a regularization of divergent convolution sums  to provide a mathematical explanation for $L$-values (non-critical in the sense of Deligne) appearing in modular graph functions  \cite{DKS2021_2}.  
\end{abstract}

\section{Introduction}

In this paper, we give a relation between  Fourier coefficients of modular forms and values of corresponding $L$-functions (not necessarily critical in the sense of Deligne) to convolution sums  
\begin{equation}\label{eq:convsum}
  \sum_{\stackrel{n_1+n_2 = n}{n_1 \in \mathbb{Z} \setminus \{ 0, n \} }}  \mathcal{Q}(n_1,n_2)  \sigma_{-r_1}(n_1) \sigma_{-r_2}(n_2),
  \end{equation}
  where $n\in\mathbb{N}$, $r_1,r_2\in\mathbb{C}$, $\mathcal{Q}$ is a combination of hypergeometric functions, and for  $x \in \mathbb{Z}\setminus \{0\}$ and $a \in \mathbb{C}$, $\displaystyle
    \sigma_{a}(x) = \sum_{\stackrel{d>0}{d | x }} d^a$ 
is the \textit{divisor function}. 
For example, we show \begin{equation}\label{eq:k=12,r1=1/3}
    \begin{split}
\sum_{ \stackrel{n_1 \in \mathbb{Z} \setminus \{ 0,n \}}{n_1+n_2=n}  } & \psi(n_1,n_2) \sigma_{-\frac{1}{3}}(n_1) \sigma_{-\frac{1}{3}}(n_2)  
\\ = &-  \frac{ n^5 \sigma_{\frac{1}{3}}(n) }{60} \left( \frac{(2 \pi)^{\frac{1}{3}} \zeta(\frac{2}{3})}{\Gamma( \frac{19}{3}) } + \frac{\zeta(\frac{4}{3}) n^{-\frac{1}{3}}  }{   (2\pi)^{\frac{1}{3}} \Gamma(\frac{17}{3}) }     \right)     
+   \frac{ 945 }{ 2^{\frac{52}{3}} \pi^{\frac{34}{3}} \Gamma(\frac{17}{3})   } \frac{L( 6,\Delta ) L( \frac{19}{3}, \Delta )}{\langle \Delta,  \Delta \rangle} \tau(n),     
    \end{split}
\end{equation}
where 
\begin{align}
\displaystyle\psi(n_1, n_2) =\frac{243 \sqrt{3}\,   }{25625600    } \left(\sqrt[3]{| n_1| }\, p(n_1,n_2)+\sqrt[3]{| n_2| } \,p(n_2,n_1)\right)
\end{align}
for $p(\alpha,\beta) =11 \alpha^5 -440\alpha^4 \beta  +2288 \alpha^3\beta^2  - 2860 \alpha^2 \beta^3 +910 \alpha\beta^4 -52 \beta^5$, $\zeta(s)$ is the Riemann zeta function, $\Gamma(s)$ is the Gamma function, and $L(s,\Delta)$ is the $L$-function of the weight 12 cuspform $\Delta(z) = \sum_{n\geq 1} \tau(n) e^{2\pi i n z}$.

Our motivation for considering convolution sums \eqref{eq:convsum} comes from computations in graviton scattering and $\mathcal{N}=4$ supersymmetric Yang-Mills theory \cite{CGPWW2021, SDK}. 
Specifically, a conjecture from string theory stated that, for certain~$\mathcal{Q}$, similar convolution sums vanish \cites{fedosova2024shifted, CGPWW2021, SDK}.
Explicitly, in 2021, Chester, Green, Pufu, Wang, and Wen conjectured that \cite[Section C.1(a)]{CGPWW2021}
\begin{equation}\label{eq:conj}
\begin{split}
\displaystyle\sum_{\stackrel{n_1, n_2 \in \mathbb{Z} \smallsetminus \{ 0\}}{n_1+n_2=n}} \varphi(n_1,n_2) & \sigma_2(n_1) \sigma_2(n_2) 
-
\left( \frac{\zeta(2) n^2}{2} + 30 \zeta'(-2)\right)\sigma_2(n) =0 ,
\end{split}
\end{equation}
where 
\begin{equation} 
\begin{split}
\varphi(n_1,n_2)= & -\frac{n_1^2}{4 n_2^2}-\frac{7 n_1}{2 n_2}-\frac{n_2^2}{4 n_1^2}-\frac{7 n_2}{2 n_1}+\frac{47}{2} 
+\left(15-\frac{30 n_1}{n_1+n_2}\right) \log \left|\frac{n_1}{n_2}\right|.
\end{split}
\end{equation}
In prior work with Danylo Radchenko, we refined and proved a generalization of this conjecture which surprisingly related convolutions sums involving even index divisor functions to Fourier coefficients of modular forms normalized by critical $L$-values \cite{FKLR}. 
However, 
our previous results were unable to give a complete account for similar anomalies appearing in other string theory computations. 
Specifically, work of Dorigoni, Kleinschmidt, and Schlotterer points towards similar identities underlying some unexplained $L$-values appearing in the theory of modular graph functions 
\cite{DKS2021_2}. The results in this paper  account for this phenomenon
by giving an exact relationship between certain shifted convolutions of divisor functions with indices in $\mathbb{C}$ to coefficients of modular forms normalized by particular $L$-values. 

Unlike the  main result in \cite{FKLR}, in this paper, the $L$-values which arise are not necessarily critical in the sense of Deligne.\footnote{ The critical values of $L(\cdot, f)$ in the sense of Deligne \cite{kontsevich2001periods} are $s=1, \ldots, k-1$ where $k$ is the weight of $f$.}  Deligne's conjectures state that critical $L$-values are expressible in terms of
algebraic numbers and certain canonical  transcendentals; however, other values are far less tractable \cite{kontsevich2001periods, diamantis2020kernels}.
In this paper we give exact identities involving not only critical $L$-values, but also values of $L$-functions evaluated at arguments which, in general, are not integers and may even be complex. These identities give new explicit closed-form expressions that evaluate to (not necessarily critical) values of 
$L$-functions of Hecke cuspforms. As a special case of our result, 
we recover terms which arise in the context of modular graph functions~\cite{DKS2021_2} using the regularization of divergent sums of the form \eqref{eq:convsum}.
As an example, we obtain that the  sum  
\begin{equation}
    \sum_{\stackrel{n_1+n_2=n}{n_1, n_2 \in \mathbb{Z} \setminus \{ 0, n\}  } } \text{sgn}(n_1) \left(n_2-n_1\right)\sigma_{-7}(n_1) \sigma_{-7}(n_2) 
\end{equation}
can be regularized and expressed in terms of $L(6,\Delta) L(13, \Delta)$ as seen in \eqref{eq:shameless_regularization} and \eqref{eq:shameless_regularization1}. 
The new ingredient in our approach is a hypergeometric regularization of divergent convolution sums (as opposed to the classical analytic regularization by factors such as $n_1^{-s}$). To our knowledge, this method has not previously appeared in the literature. The explicit result and its connection to modular graph functions are summarized in Section \ref{sec:mgf}, and the proof can be found in Section \ref{sec:regularization_of_divergent_sums}.

The main result of this investigation not only yields an extension of Theorem 1 in \cite{FKLR} to $r_1,r_2, d\in\mathbb{C}$ (not just $r_1,r_2\in2\mathbb{Z}$, $d \in \mathbb{N}$) but also encompasses  \textit{finite} convolution sums of Jacobi  referenced in \cite{motohashi1994binary} and  Diamantis and O'Sullivan in \cite[Section 4.2]{diamantis2010kernels} and \cite{o2023identities}. 
To formulate our main result, assume $n \in \mathbb{N}$ and  $n_1, n_2 \in \mathbb{Z}-\{0\}$ with $n_1+n_2 = n$. Let $G_{\pm}(n_1, n_2) = G_{d,\pm}^{(r_1, r_2)}(n_1, n_2)$ be defined by 
\begin{equation}\label{G:def_sksslldffdddd}
G_{\pm}(n_1, n_2) = \lim_{\varepsilon \to 0+} \, _2 F_1 (d+1,d+r_1+1;k;\frac{n_1+n_2}{n_1} \pm  i \varepsilon),
\end{equation} 
where $\, _2 F_1(\cdot, \cdot, \cdot; z)$ is the hypergeometric function, analytically  extended from the original domain of definition $|z|<1$ to $z \in \mathbb{C} \setminus [1, \infty)$. 
Define $\mathcal{Q}(n_1, n_2) = \mathcal{Q}^{(r_1, r_2)}_d(n_1, n_2)$ by
    \begin{equation}\label{eq:Q_convolution_form_main_theorem}
    \begin{split}
  \Gamma(k)  \left|\frac{n_1}{n}\right|^{d+1}  \mathcal{Q}(n_1,n_2) = & \left(  G_+(n_1, n_2)  - G_-(n_1, n_2)  \right) \cdot  i^{k+1} \sin( \pi r_2 / 2) \\
    + &  \left(  G_+(n_1, n_2)  + G_-(n_1, n_2) \right) \cdot \begin{cases}
       i^{k}  \cos(\pi r_2 / 2), \quad &n_1 > 0, \\
             \cos(\pi r_1 / 2), \quad &n_1 < 0
        \end{cases}
    \end{split}
\end{equation} 
and let 
\begin{equation}\label{def:Zabd}
         {Z}^{(\alpha,\beta)}_d =  \frac{ \zeta(1-\beta) }{ (2 \pi )^{-\beta}\Gamma (\alpha + \beta +d+1) \Gamma (d+\beta+1)  } +   \frac{\zeta(1+\beta)}{(2 \pi)^{\beta}\Gamma(d+1) \Gamma(d+\alpha+1)} n^{-\beta }.
\end{equation}

\begin{letterthm}
\label{thm:THEOREM-A} 
Let $n \in \mathbb{N}$, let $k$ be an even integer with $k \ge 6$ and 
let $r_1, r_2, d \in \mathbb{C}$ \ be such that 
\begin{equation}\label{eq:cond_1}
    r_1 + r_2 + 2 d + 2 = k\ \ \  \text{ and } \ \ \ \Re(r_1), \Re(r_2) \ge 0 \ \ \ \text{ and } \ \ \  \Re(d) > 0.
\end{equation}
Then, as an absolutely convergent series,
\begin{equation}\label{eq:first_main_thm_statement_sdfsdffff}
\begin{split}
    \sum_{\stackrel{n_1+n_2 = n}{n_1 \in \mathbb{Z} \setminus \{ 0, n \} }} & \mathcal{Q}(n_1,n_2)  \sigma_{-r_1}(n_1) \sigma_{-r_2}(n_2) 
   =   - i^{ k} \cdot  {Z}^{ (r_1,r_2) }_d\sigma _{-r_1}(n)    - {Z}^{(r_2,r_1)}_d\sigma_{-r_2}(n)    \\
& - i^{k} \frac{\Gamma(k-1)}{ (4 \pi )^{k-1}} \frac{1}{  (2 \pi)^{r_2} } \sum_{f \in \mathcal{F}_k}  \frac{L( d+r_2+1,f ) L( d+r_2+r_1+1, f ) }{ \Gamma(d+1) \Gamma(d+r_1+1)  \langle f,  f \rangle } \frac{a_f(n)}{n^{d+r_1+r_2}}
\end{split}
\end{equation}
where the summation runs over  $\mathcal{F}_k$ a space of normalized Hecke  cusp forms $f(z)=\sum_{n\geq 1}a_f(n)e^{2\pi i n z}$ of weight $k$, $L(\cdot, f)$ denotes the $L$-function associated to $f$, and $\langle \cdot, \cdot \rangle$ is the Petersson product. 
\end{letterthm}

    In order for the second $L$-value in \eqref{eq:first_main_thm_statement_sdfsdffff} to be critical in the sense of Deligne, one must require $d+r_1+r_2+1 \in \mathbb{Z}$, i.e.,  $k-(d+r_1+r_2+1) = d + 1 \in \mathbb{Z}$.  Additionally, for the first $L$-value to be critical, we require $d+r_1 \in \mathbb{Z}$. Interestingly, in both of these cases, $\mathcal{Q}$ becomes an elementary function (see Section~\ref{sec:Examples}). For both of the $L$-values to be critical, we require $d, r_1, r_2 \in \mathbb{Z}$. Such cases correspond to identities in \cite{FKLR} and a generalized version of  \cite{diamantis2010kernels, o2023identities}.

 The proof of the aforementioned main result in \cite{FKLR}, which only applied to the case where $r_1$ and $r_2$ are even integers, used the holomorphic projection lemma. In this paper, we consider a more general~$\mathcal{Q}$ so that similar identities can be achieved for $r_1,r_2\in\mathbb{C}$. The proof of Theorem~\ref{thm:THEOREM-A} in Section \ref{sec:proofs} uses spectral methods in the spirit of Kuznetsov and Motohashi \cites{kuznetsov1985convolution, motohashi1994binary}.
We consider sums of the form \begin{equation}\label{eq:integral_after_contour_changedfsdd4}
	\begin{split}
		\mathcalSsumdef:=  \sum_{m=1}^\infty \kappa(m) m^{-R}  \sum_{\ell=1}^\infty \frac{2 \pi }{\ell} S(-m,-n; \ell)  J_{k-1} ( \tfrac{4 \pi \sqrt{mn}}{\ell}),
	\end{split}
\end{equation} where $\kappa(m) =\sigma_{-r_1}(m)$, $S$ denotes the Kloosterman sum, and $J_{k-1}$ is the $J$-Bessel function of order $k-1$, in two different ways. On one hand, we use the Petersson trace formula to derive the contribution of the Hecke cuspforms. On the other hand, we use a Mellin inversion identity for the Bessel functions of the first kind. In this direction, once we shift the contour to be in the region of convergence for the Estermann zeta function, we are able to obtain the convolution sum and the remaining ``boundary terms" defined by $Z^{(\alpha, \beta)}_d$. 

It is worth noting that there is a more direct and constructive version of this proof outlined  in Section~\ref{sec:sketches}. While the proof in Section \ref{sec:sketches} is more elegant and intuitive, the one included in Sections~\ref{sec:petersson}--\ref{sec:conclusion} illustrates more clearly that the function $\mathcal{Q}$ defined in \eqref{eq:Q_convolution_form_main_theorem} is in fact the most general weight function of such convolution sums which  yields closed expressions involving Hecke eigenforms. A more general form of~$\mathcal{Q}$ will yield other terms involving Maa\ss\,  cuspforms and integrals of non-holomorphic Eisenstein series from the Kuznetsov trace formula.

The method illustrated in this paper is not tied to a product of two divisor functions. One could choose $\kappa(m)$ in \eqref{eq:integral_after_contour_changedfsdd4} to be the $m$-Fourier coefficients of a Hecke cusp form instead of a divisor function. Formally following the proof of Theorem~\ref{thm:THEOREM-A}, one obtains  convolution sums of the form

\begin{equation}
    \sum_{\stackrel{n_1+n_2=n}{n_1, n_2 \in \mathbb{Z} \setminus \{ 0, n\}  }}  		   \kappa(n_1) \sigma_{-r_2}(n_2)   \mathcal{W}(n_1,n_2),
\end{equation}
where $\mathcal{W}$ is a certain combination of hypergeometric functions. In this case, the resulting $L$-values would be associated to Rankin-Selberg $L$-functions.

In the remainder of this section we provide a number of examples of Theorem~\ref{thm:THEOREM-A} where $\mathcal{Q}$ can be simplified. Section \ref{sec:mgf} extends Theorem~\ref{thm:THEOREM-A} to special cases with $\text{Re}(d)<1$ resulting in Theorem~\ref{thmB}. This result uses regularization of a divergent series, and the proof of this extension is given in Section~\ref{sec:regularization_of_divergent_sums}. Theorem~\ref{thmB} gives an account for mysterious appearing of $L$-values  in the theory of modular graph functions. We provide a brief set-up of some of the previously known 
objects and results appearing in the proof of Theorem~\ref{thm:THEOREM-A} in Section \ref{sec:prelim}. Section 4 has three major components. Preliminary lemmas are proven in Section \ref{sec:preliminaryresults}. Then, for $\operatorname{Re}(r_2) > \operatorname{Re}(r_1) + 2 > 2$, the sum $\mathcalSsumdef$ as in \eqref{eq:integral_after_contour_changedfsdd4} is computed two ways: using the Petersson trace formula in Section \ref{sec:petersson} and using Mellin inversion in Section \ref{sec:InvMellin}. The proof of Theorem~\ref{thm:THEOREM-A} is completed via meromorphic continuation in Section \ref{sec:mero}.

\subsection{Examples}\label{sec:Examples}

For certain choices of $r_1, r_2, d$  the function $\mathcal Q(n_1, n_2)$ from  Theorem~\ref{thm:THEOREM-A} simplifies. These simplifications are explained in the Appendix.
Additionally, in Sections \ref{sec:d_positive_integer_examples} and \ref{sec:d+r1integer_section}, at least one of respective $L$-values appearing in the right hand side of \eqref{eq:first_main_thm_statement_sdfsdffff} is critical in the sense of Deligne.

When $d$ is an integer and  $r_1, r_2$ are even integers, the first term in \eqref{eq:Q_convolution_form_main_theorem} vanishes, and we recover Theorem~1 in \cite{FKLR} (except for the case $r_1 = r_2 = 0$ and $d=1$, that we do not consider here to avoid adding even more technicalities to the proof).

 When $d$ is an integer and $r_1, r_2$ are odd integers, the second term in \eqref{eq:Q_convolution_form_main_theorem} vanishes. Moreover, note that for the fixed value of $n$, the first term is non-zero only for finitely many $n_1$. The result subsumes  \cite[Section 4.2]{diamantis2010kernels} and \cite{o2023identities} which introduced an algorithmic approach (though not a closed-form expression) for computing such identities.

We proceed with considering specializations of the main result to certain parameters.

\subsubsection{$d$ is a positive integer}\label{sec:d_positive_integer_examples}
 When $d$ is a positive integer, $\mathcal{Q}(n_1,n_2)$ is an elementary function.
As noted above, our result subsumes those found in \cite[Section 4.2]{diamantis2010kernels} and \cite{o2023identities}:  as an example, for $r_1 = 3$, $r_2 = 5$, $d=2$, note that $k=14$ and the convolution identity becomes 
\begin{equation}
    \sum_{n_1 = 1}^{n-1}\left(  10 -55 \tfrac{n_1}{n} +66 \tfrac{n_1^2}{n^2} \right)\sigma _{3}(n_1) \sigma_{5}(n-n_1) =\frac{\sigma _3(n)-\sigma _5(n)}{24},
\end{equation}
 and for $r_1 = 1, r_2 = 3, d=1$ (thus, $k=8$) we have 
\begin{equation}\label{eq:identity_osullivanooooeo}
    \sum _{n_1=1}^{n-1} (1- 3 \tfrac{n_1}{n}) \sigma_{1}(n_1) \sigma _{3}(n-n_1) = 
    \frac{\sigma _1(n)- (6n-5)\sigma _3(n)}
    {120}.
\end{equation}
For $d=2$ and $r_1=r_2=3$ (thus, $k=12$) we have
\begin{align}
   \sum _{n_1=1}^{n-1}
\frac{\pi   (3n_1-2 n) (3 n_1-n) }{1440\,\, n^8}\sigma_{3}(n_1) \sigma _{3}(n-n_1) = -  \frac{\zeta(4) \sigma_{3}(n) }{4 \pi^{3}\Gamma(3) \Gamma(6)n^6} 
+ \frac{\Gamma(11)L( 6,\Delta) L( 9, \Delta )}{ (4 \pi )^{11} (2 \pi)^{3}\Gamma(3) \Gamma(6) \langle  \Delta,   \Delta \rangle }  \frac{\tau(n)}{n^{8}},
\end{align}
where $\Delta(z) = \sum_{n\geq 1} \tau(n) e^{2\pi i n z}$ is the weight 12 cuspform.

Note that examples for $r_1, r_2$ odd and $d \in \mathbb{N}$ are a  generalization of Niebur's formula for the Ramanujan tau function, cf. \cite[Section 4.2]{diamantis2010kernels}. For example,
when $d=2$ and $r_1=r_2=5$ (thus, $k=16$), we have 
\begin{align}
   \sum _{n_1=1}^{n-1}
\frac{\pi (-3 n^2+13 n n_1-13 n_1^2)}{1814400\, n^{12}}\sigma_{5}(n_1) \sigma _{5}(n-n_1) 
=    & - \,\frac{2\zeta(6)}{(2 \pi)^{5}\Gamma(3) \Gamma(8)} \sigma_{5}(n) n^{-10}\nonumber \\ & 
\ \ \ +   \Gamma(15) \frac{L( 8,f ) L( 13, f ) }{(2 \pi)^{5}(4 \pi )^{15} \Gamma(3) \Gamma(8)  \langle f,  f \rangle } \frac{a_f(n)}{n^{12}},
\end{align}
where $f(z)=\sum_{n\geq 1}{a_f(n)}e^{2\pi i n z}$ the weight 16 cuspform.

For $r_1=\frac{1}{2}$, $r_2=\frac{3}{2}$, $d=1$ (thus,   $k=6$) we have 
    \begin{align}
\sum_{ \stackrel{n_1 \in \mathbb{Z} \setminus \{ 0,n \}}{n_1+n_2=n}  }   \mathcal{Q}_1(n_1,n_2) \sigma_{-\frac{1}{2}}(n_1) \sigma_{-\frac{3}{2}}(n_2)  
& = \left( \frac{ \zeta(-\frac{1}{2}) }{ (2 \pi )^{-\frac{3}{2}}\Gamma (4) \Gamma (\frac{7}{2})  } +   \frac{\zeta(\frac{5}{2})}{(2 \pi)^{\frac{3}{2}}\Gamma(2) \Gamma(\frac{5}{2})} n^{-\frac{3}{2} }\right)\sigma_{-1/2}(n) \nonumber\\ 
& \  \  - 
\left( \frac{ \zeta(\frac{1}{2}) }{ (2 \pi )^{-\frac{1}{2}}\Gamma (4) \Gamma (\frac{5}{2})  } +   \frac{\zeta(\frac{3}{2})}{(2 \pi)^{\frac{1}{2}}\Gamma(2) \Gamma(\frac{7}{2})} n^{-\frac{1}{2}}\right)\sigma_{-3/2}(n),
\end{align}
 where $\mathcal{Q}_1(n_1,n_2)$ is defined by 
\begin{equation}
    \begin{split}
 \frac{(n_1+n_2)^3}{2 \sqrt{2}} \mathcal{Q}_1(n_1,n_2)=&\,  \frac{ n_1^2 n_2}{3 }+\frac{ n_1^3}{45 }+\frac{ n_2^3}{9 }-  n_1 n_2^2 + 8|n_1|^{1/2} |n_2|^{3/2} \left(\frac{n_1}{9}- \frac{n_2}{15} \right).
    \end{split}
\end{equation}

\subsubsection{$d + r_1$ is an integer}\label{sec:d+r1integer_section}
     
When $d + r_1$ is an integer (and hence, $d + r_2$ is an integer so that  $k = 2 d + r_1 + r_2 + 2 \in \mathbb{Z}$), the function $\mathcal{Q}(n_1,n_2)$ is an elementary function. 
For $d=1-i$ and $r_1=-r_2 = 1 + i$,  (and hence $k=6$), we have 
\begin{equation}
    \begin{split}
      \mathcal{Q}_2(n_1, n_2 )=  \frac{\sinh (\pi / 2)}{10 (n_1+n_2)^{3+i}}   \left(| n_1| ^{1+i} p_3(n_1,n_2)-| n_2| ^{1+i} p_3(n_2,n_1)\right),
    \end{split}
\end{equation}
for 
$
    p_3(\alpha, \beta ) = \alpha^2 + (2-6 i) \alpha \beta - 5 \beta^2
$
and corresponding sum identically vanishes:
\begin{equation}
\sum_{ \stackrel{n_1 \in \mathbb{Z} \setminus \{ 0,n \}}{n_1+n_2=n}  } \mathcal{Q}_2(n_1, n_2) \sigma_{-1-i}(n_1) \sigma_{-i-1} (n_2) = 0.
\end{equation}
For $r_1 = r_2 = i$, $d=3-i$ (thus, $k=8$) we have 
\begin{equation}
\sum_{ \stackrel{n_1 \in \mathbb{Z} \setminus \{ 0,n \}}{n_1+n_2=n}  } \mathcal{Q}_3(n_1, n_2) \sigma_{-i}(n_1) \sigma_{-i} (n_2) = -\frac{\sigma _{-i}(n)}{3}  \left(    \frac{\zeta(1-i)}{ (2 \pi)^{-i} \Gamma(i+4)  }   + \frac{\zeta(1+i)}{  (2 \pi)^i \Gamma(4-i)  } n^{-i}  \right)
\end{equation}
where 
\begin{equation}
    \begin{split}
        \mathcal{Q}_3(n_1, n_2) = -\frac{ \cosh (\pi  / 2)}{30 n^{3+i}}  \left(| n_1| ^i  p_4(n_1, n_2) +| n_2| ^i p_4(n_2, n_1 ) \right)
    \end{split}
\end{equation}
for 
$
p_4(\alpha, \beta) =     \alpha^3-(3+6 i) \alpha^2 \beta-(3-6 i) \alpha \beta^2+\beta^3.$
When $d=\frac{14}{3}$ and $r_1=r_2=\frac{1}{3}$ so that $k=12$, we arrive at \eqref{eq:k=12,r1=1/3}.

\subsubsection{$d$ is a positive half-integer and $r_1$ and $r_2$ are integers of opposite parity}

 When $d$ is a positive half-integer and $r_1$ and $r_2$ are integers of opposite parity,  $\mathcal{Q}(n_1,n_2)$ is a combination of elementary functions and elliptic integrals. Aa an example, the convolution identity for $r_1=1, r_2 = 2$, and $d=1/2$ (hence, $k=6$) can be written as follows: \begin{equation}
    \begin{split}
\mathcal{Q}_4(n_1,n_2)  &  =\delta_{n_1 > 0}  \frac{128   }{1575 \pi }  \frac{\sqrt{n_1}}{(n_1+n_2)^{7/2}} \\
        & \times \left(n_2 \left(n_1^2+74 n_1 n_2-55 n_2^2\right) q_K(n_1, n_2)  +\left(2 n_1^3+17 n_1^2 n_2-108 n_1 n_2^2+5 n_2^3\right) q_E(n_1, n_2) \right) ,
    \end{split}
\end{equation}
for $\delta_{n_1>0}$ is $1$ when $n_1>0$ and $0$ otherwise and 
\begin{equation}
\begin{split}
    q_K(n_1, n_2) &= \frac{1}{2} \lim_{\varepsilon \to 0+} ( K(1+\tfrac{n_2}{n_1}+i \varepsilon) +K(1+\tfrac{n_2}{n_1}-i \varepsilon)), \\ 
    q_E(n_1, n_2) &= \frac{1}{2} \lim_{\varepsilon \to 0+} ( E(1+\tfrac{n_2}{n_1}+i \varepsilon) +E(1+\tfrac{n_2}{n_1}-i \varepsilon)), 
\end{split}
    \end{equation}
 where $K$ and $E$ denote the complete elliptic integral of the first and second kind respectively.
The convolution identity reads 
\begin{equation}
    \begin{split}
    \sum_{ \stackrel{n_1 \in \mathbb{Z}_+ \setminus \{ 0,n \}}{n_1+n_2=n}  } & \mathcal{Q}_4(n_1, n_2) \sigma_{-1}(n_1) \sigma_{-2} (n_2) =  \sigma _{-1}(n) \left(\frac{2 \zeta (3)}{3 \pi ^3 n^2}-\frac{128 \pi }{4725}\right)+\frac{4 (16 n-7) \sigma _{-2}(n)}{315 n}.
    \end{split}
\end{equation}
\subsection{Divergent series and modular graph functions}\label{sec:mgf}

In Section \ref{sec:regularization_of_divergent_sums}, we regularize divergent sums of the form 
\begin{equation}\label{eq:Qagain}
 \sum_{\stackrel{n_1+n_2 = n}{n_1 \in \mathbb{Z} \setminus \{ 0, n \} }}  \mathcal{Q}(n_1,n_2)  \sigma_{-r_1}(n_1) \sigma_{-r_2}(n_2).\end{equation}
Our interest in these sums is motivated by \textit{modular graph functions}.  A notable example arises from studying solutions \(f\) of the differential equation  
\begin{equation}\label{eq:DE}
    (\Delta - R(R+1)) f(z) \;=\; E_a(z)\,E_b(z) 
    \qquad \text{on } SL_2(\mathbb{Z}) \backslash SL_2(\mathbb{R}) / SO_2(\mathbb{R}),
\end{equation}
where \(a,b,R \in \mathbb{Z}\), \(\Delta = y^2(\partial_x^2+\partial_y^2)\) is the Laplace-Beltrami operator, 
and \(E_a(z), E_b(z)\) are non-holomorphic Eisenstein series. Such equations are central objects in the theory of modular graph functions \cites{DKS2021_1, DKS2021_2}.
 It is expected that, based on a similar observation when $a$ and $b$ are half-integers \cites{CGPWW2021, FKLR}, homogeneous parts of these solutions correspond to divisor sums of the form \eqref{eq:Qagain} with  
\begin{equation}\label{eq:modular_graph_numbers1}
    k = 2 R + 2, \quad r_1 = 2 a - 1, \quad r_2 = 2 b - 1.
\end{equation}
In order to formulate the next theorem, we 
introduce \textit{Jacobi polynomials}, that are defined for $d \in \mathbb{N}_0$ and $\alpha, \beta > -1$ \cite[Table 18.3.1]{NIST} by 
\begin{equation}
P^{(\alpha, \beta)}_d(z):= \frac{\Gamma(\alpha +d+1)}{d!\,\Gamma(\alpha+\beta+d+1)}\sum_{m=0}^d {d\choose m}\frac{\Gamma(\alpha+\beta+d+m+1)}{\Gamma(\alpha+m+1)}\left(\frac{z-1}{2}\right)^m.
\end{equation}

The proof of the following theorem can be found in Section \ref{sec:regularization_of_divergent_sums}.

\begin{letterthm}\label{thmB} Let $r_1, r_2$ be odd integers with $r_1 \ge r_2\geq 3$ and $d \in \mathbb{Z}$ such that 
$
    -r_2 + 1 \le d \le -2
$ and 
 \begin{equation}
d\neq -r_1-1 \text{ or } -r_1-r_2-1.     
 \end{equation}
Let $n \in \mathbb{N}$ and 
let $n_1, n_2 \in \mathbb{Z}$ with $n_1+n_2 = n$ and $n_1, n_2 \neq 0$. Let $\mathcal{P}(n_1,n_2)$ be defined by  
\begin{equation}\label{eq:P}
    \begin{split}
  \Gamma(k-d-1) \cdot \mathcal{P}(n_1,n_2)  =& \   2  \pi i^{k+r_2+1}   \left( \frac{n_2}{n}  \right)^{r_2}  P_{d+r_1}^{(r_2,-r_1)}\left(2\tfrac{ n_1}{n}-1\right) \cdot \begin{cases}
    1  , & 0 < n_j < n \text{ for } j=1,2, \\
    0, & \text{otherwise}.
\end{cases}
\\
&  +      \pi i^{ r_1+1 }   \left(\frac{n}{n_1}\right)^{ d+1}   P^{  (k-1, -r_2) }_{-d-1} (1-2 \tfrac{n}{n_1}) \cdot \begin{cases}
    1  , & n_1 > 0, \\
    -1, & n_1 < 0.
\end{cases}
    \end{split}
\end{equation}
 Then   the divergent sum  \begin{equation}
  \sum_{ \stackrel{n_1 \in \mathbb{Z} \setminus \{ 0,n \}}{n_1+n_2=n}  } \sigma_{-r_1}(n_1) \sigma_{-r_2 }(n_2)  \mathcal{P}(n_1, n_2), \end{equation}
admits a regularization given by 
\begin{equation}
\begin{split}
     & -  i^{ k} \cdot  \mathcal{Z}^{r_1,r_2}_d\sigma _{-r_1}(n)    - \mathcal{Z}^{r_2,r_1}_d\sigma_{-r_2}(n) \\
     &  \ \ \ \ \ \  -\frac{i^k 2^{2 - 2 k - r_2} \pi ^{1-k-r_2} \Gamma (k-1) }{\Gamma \left(1+d+r_1\right)} \sum_{f \in \mathcal{F}_k} \frac{L\left(1 + d + r_2,f\right) L\left(1 + d + r_1 + r_2,f\right)}{\langle  f,  f \rangle }   \frac{a_f(n)}{ n^{d+r_1+r_2}}, \\
\end{split}
\end{equation}
 where the summation is over  $\mathcal{F}_k$ a space of normalized Hecke  cusp forms of weight $k$, $L(\cdot, f)$ denotes the $L$-function associated to $f$, $\langle \cdot, \cdot \rangle$ is the Petersson product, and 
\begin{equation}
    \begin{split}
\mathcal{Z}^{r_1,r_2}_d =    \frac{(-1)^{d+1} (2 \pi )^{r_2} }{\Gamma (1+d+r_2) \Gamma (1+d+r_1+r_2 )  \Gamma(-d) }  \zeta'(1-r_2) + \frac{(2 \pi )^{-r_2} \zeta (1+r_2 )}{\Gamma (1+d+r_1 )} n^{-r_2}.
    \end{split}
\end{equation}
\end{letterthm}
\vspace{.2cm}

For a fixed value of $n$,  the first term in~$\mathcal{P}(n_1,n_2)$ is equal to zero except for finitely many~$n_1$. We can subtract their sums  from the convolution identity and obtain for the case $a=b=4, d=-2$, that the theorem above regularizes the sum 
\begin{equation}\label{eq:shameless_regularization}
    \sum_{\stackrel{n_1+n_2=n}{n_1, n_2 \in \mathbb{Z} \setminus \{ 0, n\}  } } \text{sgn}(n_1)\left(\frac{n_2-n_1}{n_1+n_2}\right)\sigma_{-7}(n_1) \sigma_{-7}(n_2) ,
\end{equation}
and this regularized sum is equal to 
\begin{equation}\label{eq:shameless_regularization1}
\begin{split}
 2& \sum_{ n_1 =1} ^n \frac{ (n-n_1)^7 }{n^{12}}(14 n^3 n_1^2+28 n^2 n_1^3+5 n^4 n_1+n^5+42 n n_1^4+42 n_1^5)\,\sigma_{-7}(n_1) \sigma_{-7 }(n_2) \\ 
     & \ \   +  \left(\frac{33 }{30 n^7}-\frac{32\pi^6 \zeta '(-6)}{90}\right)\sigma _{-7}(n)    +\frac{ \Gamma (11) \Gamma(13)}{3\cdot 2^{30}\pi ^{19}\Gamma \left(6\right)} \frac{L\left(6,\Delta\right) L\left(13,\Delta\right)}{\langle \Delta, \Delta \rangle }   \frac{\tau(n)}{n^{12}}. 
\end{split}
\end{equation}
This regularized sum corresponds to  the values found in (3.11) of \cite{DKS2021_2} and (4.22) of \cite{DGW}. 

For $a=7$, $b=6$ and $d=-4$, 
\begin{equation}\label{eq:shameless_regularization2}
    \sum_{\stackrel{n_1+n_2=n}{n_1, n_2 \in \mathbb{Z} \setminus \{ 0, n\}  } } \frac{ \left(6 n_1^3-18 n_1^2 n_2+22 n_1 n_2^2-11 n_2^3\right)}{ (n_1+n_2)^3}\sigma_{-13}(n_1) \sigma_{-11}(n_2) ,
\end{equation}
and this regularized sum is equal to
\begin{equation}
\begin{split}
   2    \sum_{ n_1 =1} ^n&\frac{  (n-n_1)^{11}\, p(n,n_1)}{n^{20}}\sigma_{-13}(n_1) \sigma_{-11 }(n_2)\\ &  +
   \left(\frac{223193  }{1260 n^{11}}-\frac{16}{4725} \pi^{10} \zeta '(-10)\right)\sigma_{-13}(n)
   + \left(\frac{323 }{n^{13}}-\frac{8 \pi^{12} \zeta '(-12)}{42525}\right)\sigma_{-11}(n)\\ & 
   -\frac{ \Gamma (17) \Gamma(21)}{5\cdot 2^{47}\pi ^{29}\Gamma \left(10\right)} \frac{L\left(8,f\right) L\left(21,f\right)}{\langle f, f \rangle }   \frac{a_f(n)}{n^{20}},
\end{split}
\end{equation}
where $p(n,n_1) = 11 n^9 + 66 n^8 n_1 + 216 n^7 n_1^2 + 504 n^6 n_1^3 + 924 n^5 n_1^4 + 
 1386 n^4 n_1^5 + 1716 n^3 n_1^6 + 1716 n^2 n_1^7 + 1287 n n_1^8 + 
 572 n_1^9$ and  $f(z)=\sum_{n\geq 1}{a_f(n)}e^{2\pi i n z}$ the weight 18 cuspform.

Theorem~\ref{thm:THEOREM-A} can be extended to any $\text{Re}(d)<1$ with $d\neq -1, -r_1-1, -r_2-1, -r_1-r_2-1$ using a similar regularization argument as in Section \ref{sec:regularization_of_divergent_sums}. However, in this more general setting $\mathcal{Q}$ may not simplify to a polynomial as in \eqref{eq:P}.

Additionally, it may be relevant to note that performing a similar regularization of the main result in \cite{FKLR} will not be sufficient to arrive at Theorem~\ref{thmB}.
Specifically, one can see this from the fact that both terms from  \eqref{eq:Q_convolution_form_main_theorem}  contribute to $ \mathcal{P}(n_1, n_2) $ despite that $r_1$ and $r_2$ are odd (in the contrast to the situation when the sum was absolutely convergent).

\section{Preliminaries}\label{sec:prelim}

For $m, n \in \mathbb{Z} \setminus \{ 0 \}$ and $\ell \in \mathbb{N}$, the \textit{Kloosterman sum} is \cite[(1.1.7)]{motohashi1997spectral} \begin{equation}\label{eq:Kloost_def}
        S(m,n;\ell) = \sum_{\stackrel{h=1}{(h,\ell)=1}}^\ell e(\tfrac{1}{\ell} (mh+n \bar{h})),
    \end{equation}
    where $\bar{h}$ is inverse of $h$ modulo $\ell$ and 
     \begin{equation}\label{def:exponential_short}
     e(  \cdot ) = \exp(2 \pi i (\cdot) ).          
     \end{equation}

Let $\{ \kappa(m)\}_{m \in \mathbb{N}_0}$ be a sequence of complex numbers, let $J$ denote the Bessel function of the first type and let $S$ denote the Kloosterman sum as in \eqref{eq:Kloost_def}. We consider the sum \eqref{eq:integral_after_contour_changedfsdd4}, 
\begin{equation}
	\begin{split}
		\mathcalSsumdef:=  \sum_{m=1}^\infty \kappa(m) m^{-R}  \sum_{\ell=1}^\infty \frac{2 \pi }{\ell} S(-m,-n; \ell)  J_{k-1} ( \tfrac{4 \pi \sqrt{mn}}{\ell}),
	\end{split}
\end{equation}
for some $R \in \mathbb{C}$ that we will define later. 
On the one hand, $\mathcalSsumdef$ can be rewritten as a convolution sum (see Section~\ref{sec:InvMellin}). On the other hand, the innermost sum can be expressed via the Petersson trace formula \cite[Theorem 9.6]{iwaniec2021spectral}:
\begin{lemma} [Petersson trace formula]\label{lem:petersson_trace_formula}  Let $m,n  \in \mathbb{N}$ and 
$k$ is an even integer with $k\ge 4$. Then  we have the absolutely convergent series 
\begin{equation}
\begin{split}
      \sum_{\ell=1}^\infty  \frac{2 \pi }{\ell} &S(-m,-n; \ell)  J_{k-1} ( \tfrac{4 \pi \sqrt{mn}}{\ell})= (-1)^{\tfrac{k}{2}} \frac{\Gamma(k-1)}{ (4 \pi \sqrt{mn})^{k-1} } \sum_{f \in \mathcal{F}_k}  \frac{a_f(n) \overline{a_f(m)}}{\langle \bar f, \bar f \rangle }  - (-1)^{\tfrac{k}{2}} \delta_{m=n},
\end{split}
\end{equation}
where $\mathcal{F}_k$ is a space of normalized Hecke cusp forms $f(z)=\sum_{n\geq 1}a_f(n)e^{2\pi i n z}$ of weight $k$ and $\delta_{m=n}$ is the Kronecker delta function.
\end{lemma}

\subsection{Divisor function}
\newcommand{\kestermann}{v}
In \cite{ramanujan1918certain}, Ramanujan expressed the divisor function  via an infinite sum. Namely, for $\Re(r_2) >  0$ and $x \in \mathbb{N}$, he wrote 
\begin{equation}\label{eq:ramanujan_identity}
		\sigma_{-r_2}(  x  )   = 
        \zeta(1+r_2)   \sum_{\ell = 1}^\infty \ell^{-r_2 - 1} c_\ell (x),
	\end{equation}
where $c_\ell $ is given by 
	\begin{equation}\label{def:ce}
		c_\ell (x) = \sumk  e( \tfrac{\kestermann x}{\ell}).
	\end{equation}    
Recall 
\begin{equation}\label{eq:sum_using_euler_totient_function}
\sum_{\ell=1}^\infty \varphi(\ell) \ell^{-s} =\zeta(s-1)/\zeta(s) ,   
\end{equation}
for the Euler's totient  function $\varphi$. For $x=0$ and $\Re(r_2)>1$, the right hand side of \eqref{eq:ramanujan_identity} is well-defined as absolutely convergent series, and, using \eqref{eq:sum_using_euler_totient_function}, 
we can rewrite it as 
\begin{equation}\label{eq:sigma0}
    \zeta(1+r_2 ) \sum_{\ell=1}^\infty \varphi(\ell) \ell^{-r_2-1} 
    = \zeta(r_2).
\end{equation}
Moreover, we will need the following estimate for $\Re(r_1) > 0$: 
\begin{equation}\label{eq:divisor_estimate}
\sigma_{-r_1}(m)=o(m^{\varepsilon})    , \quad m \to \infty ,
\end{equation}
for any $\varepsilon>0$.
\subsection{Estermann zeta function}
Let $a \in \mathbb{C}$ and let $\kestermann$ and $\ell$ be co-prime integers such that $\ell \ge 1$; note that $\kestermann$ is not necessarily positive. The \textit{Estermann zeta function} is defined for $\Re(s) > \max \{ 1, \Re(a) + 1 \}$ as
	\begin{equation}\label{def:estermann}
		E(s; \kestermann/\ell, a) := \sum_{n=1}^\infty \sigma_{a}(n) e( \tfrac{\kestermann n}{\ell}) n^{-s},
	\end{equation}
    for $e(\cdot)$ as in \eqref{def:exponential_short}.   The Estermann zeta function admits the following representation (see \cite{ishibashi1995} and references therein) via Hurwitz zeta functions:
\begin{equation}\label{eq:representation_of_estermann_via_Hurwitz_zeta}
     \begin{split}
         E\left(s; \kestermann/\ell, a\right)=\ell^{a-2 s} \sum_{m_1, m_2=1}^\ell e(\frac{ \kestermann}{\ell} m_1 m_2 ) \zeta(s-a, \frac{m_1}{\ell}) \zeta(s, \frac{m_2}{\ell})
     \end{split}
 \end{equation} 
 where for $\text{Re}(s)>1$ and $a\neq 0, -1, -2,\dots$, the Hurwitz zeta function is defined as  \begin{equation}\zeta(s,a):= \sum_{n=0}^\infty\frac{1}{(n+1)^s}.\end{equation} 
 
As a consequence of \eqref{eq:representation_of_estermann_via_Hurwitz_zeta}, for $a \neq 0$, the Estermann zeta function $E(s; \kestermann/\ell, a)$ admits a meromorphic continuation to $s \in \mathbb{C}$ with poles possible only at $s=1$ and $a + 1$ with 
\begin{equation}\label{eq:estermann_zero}
     \begin{split}
         \text{Res}_{s=1} E(s ; \kestermann / \ell, a )  &= \zeta(1-a) \ell^{a-1}, \\
         \text{Res}_{s=1+a} E(s ; \kestermann / \ell, a )  &= \zeta(1+a) \ell^{-a-1}.
     \end{split}
\end{equation}
Moreover, the Estermann zeta function satisfies the functional equation
\begin{equation}\label{eq:functional_equation_estermann_zeta}
    \begin{split}
E\left(s ; \kestermann / \ell, a \right)= & \frac{1}{\pi}\Big(\frac{\ell}{2 \pi}\Big)^{1+a-2 s} \Gamma(1-s) \Gamma(1+a-s) \\
& \times \Big(\cos(\frac{\pi a}{2}) E(1+a-s ; \bar{\kestermann} / \ell, a)-\cos (\pi s-\frac{\pi a}{2}) E(1+a-s ;-\bar{\kestermann} / \ell, a) \Big), 
    \end{split}
\end{equation}
where $\bar{\kestermann}$ is the inverse of $\kestermann$ modulo $\ell$ (see \cite[(2.4)]{kiuchi1987exponential}).

\subsection{$L$-functions}

For a Hecke cusp form $f$ of weight $k$, denote by $a_f(n)$ its $n$-th Fourier coefficient, and let 
\begin{equation}
    L(s,f) = \sum_{m=1}^\infty a_f(m) m^{-s}.
\end{equation}
For any $\delta>0$ and $m \in \mathbb{N}$ we have the following estimate on its $m$-th Fourier coefficient \cite[(3.1.22)]{motohashi1997spectral}  
\begin{equation}
|a_f(m)|\ll m^{\frac{1}{4}+\frac{k-1}{2}+\delta},
\end{equation}
 where the implied constant depends only on $\delta$. Moreover, for $\Re(\alpha)<0$ that for any $\varepsilon>0$, we have   
 \begin{equation}
    | a_f(m)  \sigma_{\alpha}(m) m^{-s}| \ll m^{\frac{1}{4}+ \frac{k-1}{2}+\delta-\Re(s) + \varepsilon},
\end{equation}
where the implied constant depends only on $\delta$ and $\varepsilon$. Hence, $        \sum_{m=1}^\infty a_f(m) \sigma_\alpha(m)  m^{-s}$ absolutely converges for $\Re(s) > \frac{5}{4} + \frac{k-1}{2}+\delta$, and  
we have by  \cite[(3.2.7)]{motohashi1997spectral} and  \cite[proof of Proposition 2.1 and (2.11)]{diamantis2010kernels}
\begin{equation}
    \begin{split}
        \sum_{m=1}^\infty a_f(m) \sigma_\alpha(m)  m^{-s} = \frac{L(s,f) L(s-\alpha,f)}{\zeta(2 \cdot (s-\frac{k-1}{2}) - \alpha )}.
    \end{split}
\end{equation}
As a consequence, for $\Re (r_2-r_1) > 3/2$,
\begin{equation}\label{eq:Hecke_sum_product_of_Lvalues}
    \begin{split}
        \sum_{m=1}^\infty    a_f(m) \sigma_{-r_1}(m) m^{-d-r_2-1} &= \frac{L( d+r_2+1, f) L( d+r_2+r_1+1,f ) }{\zeta( r_2 + 1 )}.
    \end{split}
\end{equation}

\subsection{Bessel functions of the first type}
Denote by \(J_\eta(x)\) the Bessel function of the first kind, defined in \cite[(10.2.2)]{NIST}. For \(\eta \in \mathbb{C} \setminus \mathbb{Z}_{<0}\) and \(x > 0\), the function \(J_\eta(x)\) satisfies 
  \begin{equation}\label{asymp:besselj_skdjflksjdlfjslkdf}
    \begin{split}
      |  J_{\eta} ( x) | \ll
      \begin{cases}
      x^{ \Re (\eta)}, \quad &x \to +  0,\\
    x^{-1/2}, \quad &x \to +\infty  
      \end{cases}
        \end{split}
    \end{equation}
     (cf. \cite[10.7.3, 10.7.8]{NIST}).
As an elementary consequence, we deduce that for $x>0$,   $\Re(r_2) > \Re(r_1) + 1 > 1$, $\Re (d) > 0$, and $0 < \varepsilon < 1/2,$  we have 
    \begin{equation}\label{eq:J_ineq}
     |   J_{2d+r_1+r_2+1} ( x ) | \ll  x^{ \Re (r_2-r_1)-1-\varepsilon}.
    \end{equation}

Moreover, we need the following  reformulation of \cite[10.9.22]{NIST} which gives the Mellin transform of Bessel $J$-functions:
\begin{prop}
Let $r_1, r_2, d \in \mathbb{C}$ satisfy the conditions \eqref{eq:cond_1} and \eqref{cond_2a}. Then for $ \Re (\frac{1-r_1-r_2}{2} ) < c < d+1$ and $x>0$ we have 
\begin{equation}\label{eq:gamma_mellin_inverstion}
	\begin{split}
J_{2d+r_1+r_2+1} ( x) = 		\frac{1}{2 \pi i} \int_{c-i \infty}^{c+i \infty}   \frac{   \Gamma(d-s+1)  }{\Gamma (d+s+r_1+r_2+1)  }       x^{2s+r_1+r_2-1} ds .
	\end{split}
\end{equation} 
\end{prop}

\subsection{Hypergeometric functions}

When writing $\, _2F_1(\cdot, \cdot, \cdot; z)$ with $z \in \mathbb{C} \setminus [1, \infty)$, we mean its principle branch, that is, the meromorphic continuation from $|z|<1$ (defined by the absolutely convergent sum \cite[15.2.1]{NIST}) to $z \in \mathbb{C} \setminus [1, \infty)$. For $x \in \mathbb{R}$, we denote 
\begin{equation}\label{eq:hypergeom_above_and_below_the_split}
    \, _2 F_1 (\cdot, \cdot, \cdot; x \pm i0)  = \lim_{  \stackrel{\varepsilon \to 0}{\varepsilon>0} }  \, _2 F_1 (\cdot, \cdot, \cdot; x \pm i \varepsilon).
\end{equation}
For $x \ge 1$, we write
\begin{equation}\label{def:2f1_positionevee}
    \, _2 F_1 (\cdot, \cdot, \cdot; x)  = \frac{    \, _2 F_1 (\cdot, \cdot, \cdot; x + i0) +     \, _2 F_1 (\cdot, \cdot, \cdot; x - i0)}{2}.
\end{equation}

\begin{remark}\label{rem:remark_on_difference_between_cuts} 
For computational convenience we may rewrite the second term on the right hand side of \eqref{eq:Q_convolution_form_main_theorem}. Let $n \in \mathbb{N}$ and let $n_1, n_2 \in \mathbb{Z}$ such that $n_1+n_2 = n$.
Substituting $a = d+1$, $b=d+r_1+1$, $c = k = 2d+2+r_1+r_2$ to \cite[15.2.3]{NIST} and using the symmetry in the first two components, we have for $G_\pm(n_1, n_2)$ from \eqref{G:def_sksslldffdddd} that 
    \begin{equation}
    \begin{split}
      &G_+(n_1, n_2) - G_-(n_1,n_2)  = \\
      &= \delta_{ n_1, n_2 > 0  } \, \frac{ \Gamma(k) }{\Gamma(r_2+1)} \left[  2  \pi i \frac{  (\frac{n}{n_1}-1)^{r_2} }{\Gamma (d+1) \Gamma (d+r_1+1)} \, _2F_1(k-d-1,d+r_2+1;r_2+1;1-\frac{n}{n_1})   \right],
    \end{split}
    \end{equation}
    where $\delta_{ n_1, n_2 > 0  }$ is $1$ for $n_1 > 0$ and $n_2 > 0$ and is equal to zero otherwise.
\end{remark}

As a consequence of \cite[Volume III, 2.21.1.1]{prud}, we obtain the following lemma  which gives the Mellin transform of hypergeometric functions: 
{\lem\label{lem:mellin1} Let $C$ be  a vertical contour such that $0<\Re(s)<\Re(d+1),\Re(d+r_1+1)$, let $w, t \in \mathbb{C}$ with $\left|\arg\left(\frac{w}{t}\right)\right|<\pi$, then 
\begin{equation} \label{eq:mellin_transform_hypergeometric_function}
\begin{split}
  \frac{1}{2 \pi i} &\int_C \frac{\Gamma (d-s+1) \Gamma (s)  \Gamma (r_1+s)}{\Gamma (d+r_1+r_2+s+1)} w^{-d+s-1} t^{-s} ds  \\
    &=  t^{-1-d} \frac{\Gamma (d+1) \Gamma (d+r_1+1)}{  \Gamma(2 d+r_1+r_2+2) }  \, _2 F_1\left(d+1,d+r_1+1;2 d+r_1+r_2+2;-\frac{w}{t}\right).
\end{split}
\end{equation} 
}

\section{Proofs}\label{sec:proofs}

We first  prove a restricted version of Theorem~\ref{thm:THEOREM-A}, i.e., Theorem \ref{thm:THEOREM-A-precursor}. In Section \ref{sec:mero}, we extend the result to obtain the full statement via the meromorphic continuation.

{\thm \label{thm:THEOREM-A-precursor}
Theorem~\ref{thm:THEOREM-A} holds under the additional assumption 

\begin{equation}\label{cond_2a}
\operatorname{Re}(r_2) > \operatorname{Re}(r_1) + 2 > 2.
\end{equation}
}

 To prove Theorem~\ref{thm:THEOREM-A-precursor}, in Section \ref{sec:petersson}, we evaluate 
	$\mathcalSsumdef$ as in \eqref{eq:integral_after_contour_changedfsdd4} using the Petersson trace formula to get a sum over Hecke cusp forms.
 In Section \ref{sec:InvMellin}, we express the $J$-Bessel function in terms of the inverse Mellin transform to get the convolution sum \eqref{eq:convsum} and remaining ``boundary terms" $Z$ as in \eqref{def:Zabd}. We then set the resulting calculations equal.

  A key step in the proof is rewriting certain functions using their Mellin transforms and shifting the integration contour. In Section \ref{sec:preliminaryresults}, we introduce a related integrand $\mathcal{E}(s)$, analyze its poles in a strip, and study its behavior in vertical strips to justify the contour shift later in the proof.
\subsection{Preliminary Results}\label{sec:preliminaryresults}

\subsubsection{Definition of $\mathcal{E}(s)$ and its poles in the strip $\Re(s) < \Re (d)+1$}\label{sec:poles_of_E_sdklfjfjfskdkdkkk3}

For for $r_1, r_2, d$ as in \eqref{eq:cond_1} and \eqref{cond_2a}, we define  
\begin{equation}\label{eq:gammaeps}\mathcal{E}(s):=\frac{   \Gamma(d-s+1)  }{\Gamma (d+s+r_1+r_2+1)  }      \left( \frac{2 \pi \sqrt{n}}{\ell} \right)^{2s+r_1+r_2-1} E(1-s-r_1,-\bar{\kestermann}/\ell,-r_1).\end{equation}
Note that the factor $\frac{   \Gamma(d-s+1)  }{\Gamma (d+s+r_1+r_2+1)  }$ does not have poles in the half-plane $\Re(s) < d+1$. 
On the other hand, from \eqref{eq:estermann_zero}, $E(1-s-r_1,-\bar{\kestermann}/\ell,-r_1)$  has poles at 
$s=-r_1$ and  $s=0$. The condition \eqref{cond_2a} implies $r_1 \neq 0$, and hence for $\Re(s)<d+1$, the function $\mathcal{E}(s)$, defined by \eqref{eq:gammaeps},  
 has  simple poles at $s=0$ and $s=-r_1 $  with residues
\begin{equation}\label{eq:pole0}
\begin{split}
\text{Res}_{s=0} \mathcal{E}(s) 
&= -\frac{   \Gamma(d+1)  }{\Gamma (d+r_1+r_2+1)  }      \left(2 \pi \sqrt{n} \right)^{r_1+r_2-1}\ell^{-r_2} \zeta(1-r_1)
\end{split}    
\end{equation}
and
\begin{equation}\label{eq:pole-r1}
    \begin{split}
\text{Res}_{s=-r_1} \mathcal{E}(s) 
 =- \frac{   \Gamma(d+r_1+1)  }{\Gamma (d+r_2+1)  }      \left( 2 \pi \sqrt{n}\right)^{-r_1+r_2-1}   \ell^{-r_2}\zeta(1+r_1).     
    \end{split}
\end{equation}

\subsubsection{Bounds of Estermann zeta function and $\mathcal{E}(s)$ in vertical strips}

Let $s=\sigma+i\tau$ and $0<\delta<1$. By \cite[Theorem 12.23]{apostol}, for $|\tau|\geq 1$ and  $-m-\delta < \sigma<-m+\delta$ where   $m\in\mathbb{N}_0$, the Hurwitz zeta function has the following bound:
 \begin{equation}
  |  \zeta(\sigma + i \tau, \cdot ) |   \lesssim |\tau|^{m+1+\delta}.
  \end{equation} 
Assume $\Re(s)=\sigma >-\Re(r_1)-\delta$ for $0<\delta<\max\{1,\{\Re(r_1)\}\}$ then there exists $\widetilde\delta \in \mathbb{R}$ such that  $0<\widetilde\delta<1$ and $-\lfloor \Re(r_1)\rfloor-\widetilde\delta<\sigma$ and thus
\begin{equation}
\begin{split}
  |  \zeta(s, \cdot ) | &  \lesssim |\tau|^{\lfloor{\Re(r_1)}\rfloor+1+\widetilde\delta}  \quad\text{and}\quad   
  |  \zeta(s+r_1, \cdot ) |  \lesssim |\tau|^{1+\widetilde\delta + \{ \Re(r_1)\}}.
\end{split}
\end{equation}
By \eqref{eq:representation_of_estermann_via_Hurwitz_zeta}, for $s \in \mathbb{C}$ with $\Re(s) > - \Re(r_1)-\delta$, we have 
\begin{equation}\label{eq:elemetary_estimate_on_Estermann_zeta}
|    E(s; \cdot  , -r_1) | \lesssim |\tau|^{\Re(r_1)+2+2 \widetilde\delta}.
\end{equation}

Now we proceed to the evaluation of $\mathcal{E}$ in vertical strips. In order to do so, we rewrite the functional equation for the Estermann zeta function, \eqref{eq:functional_equation_estermann_zeta}, as 
\begin{equation}\label{eq:FE}
	\begin{split}
&	 E(1-r_1-s ;  -\bar{\kestermann} / \ell, -r_1) \\
& \ \ \ \ \ \ \ \ \  = \frac{ \Gamma(s)      }{ \Gamma(1-r_1-s)  }  \Big(\frac{\ell}{2\pi}\Big)^{2 s+r_1-1} \left(  \frac{\cos(\pi s + \frac{\pi r_1}{2}) E(s;\kestermann/\ell,-r_1) + \cos(- \pi r_1/2 ) E(s;-\kestermann/\ell,-r_1)  }{\sin(\pi (r_1+s))} \right).
	\end{split}
\end{equation}
Substituting  \eqref{eq:FE} into $\mathcal{E}(s)$,
we obtain 
\begin{equation}\label{eq:E_funeq}
	\begin{split}
		\mathcal{E}(s) 
          &   = 		   \left( \frac{2 \pi }{\ell} \right)^{r_2}    \left(  \sqrt{n}  \right)^{r_1+r_2-1}    \frac{   \Gamma(d-s+1)   \Gamma(s)  }{\Gamma (d+s+r_1+r_2+1) \Gamma(1-r_1-s)  } \\
&      \ \ \ \ \ \ \ \ \  \ \ \ \ \ \  \ \ \ \ \ \  \times       \left(  \frac{\cos(\pi s + \frac{\pi r_1}{2}) E(s;\kestermann/\ell,-r_1) + \cos(- \pi r_1/2 ) E(s;-\kestermann/\ell,-r_1)  }{\sin(\pi (r_1+s))} \right) n^{s} .
    \end{split}
\end{equation}
As a corollary to \eqref{eq:E_funeq} together with the estimates of the Estermann zeta function, we obtain the following lemma:
\begin{lemma}\label{lem:contour} Let $s=\sigma+i\tau$. For all $\delta>0$ and fixed $\sigma > -\Re(r_1)- \delta$,
\begin{equation}\label{eq:asymp_Omega}
\begin{split}
    |\mathcal{E}(s) |   \lesssim |\tau|^{{\text{Re}(r_1-r_2)+1+2\delta}}, \quad \tau \to \pm \infty.
\end{split}
\end{equation}
In particular, for $\text{Re}(r_2-r_1) > 1$, we get $|\mathcal{E}(s)| \to 0$ as $\tau\to \pm \infty$. 
\end{lemma}
\begin{proof}
For $|\operatorname{ph} s| \le \pi - \delta $, where $\operatorname{ph} s$ denotes the phase of $s$,  by \cite[5.11.12]{NIST} we have 
\begin{equation} \left|\frac{   \Gamma(d-s+1)   \Gamma(s)  }{\Gamma(d+s+r_1+r_2+1) \Gamma(1-r_1-s)  }  \right|\sim |s|^{-\text{Re}(r_2)-1}.\end{equation}
Moreover, for $s = \sigma + i \tau$ with fixed $\sigma$ and  $\tau  \to \pm \infty$, we have 
$\cot(\pi s) = O( 1)$
and $\csc(\pi s) = O( 1)$.
From~\eqref{eq:elemetary_estimate_on_Estermann_zeta} we have a bound on the Estermann zeta function in vertical strips: 
for  $\sigma > -\Re(r_1)- \delta$ where $\delta>0$, we have 
\begin{equation}
|    E(s; \cdot  , -r_1) | \lesssim |\tau|^{\text{Re}(r_1)+2+2 \delta},
\end{equation}
and so by \eqref{eq:E_funeq}, 
\begin{align}\label{eq:estimate_on_E_sasdee}
 |\mathcal{E}(s) |\lesssim  |s|^{-\text{Re}(r_2)-1}|\tau|^{\text{Re}(r_1)+2+2 \delta}n^s.
\end{align}
For fixed $\sigma$, the right hand side of \eqref{eq:estimate_on_E_sasdee} can be evaluated as 
\begin{equation}|\tau|^{-\Re(r_2)-1}|\tau|^{\text{Re}(r_1)+2+2 \delta}    = |\tau|^{\text{Re}(r_1-r_2)+1+2\delta}.\end{equation} 
\end{proof}
\begin{remark}
In order to obtain \eqref{eq:asymp_Omega}, we first applied the functional equation to the Estermann zeta function, and only then evaluated it using estimates for the Hurwitz zeta function instead of working directly with the Hurwitz zeta function. The approach above yields more precise estimates, as it uses information about the Estermann zeta function beyond its expression as a sum of products of Hurwitz zeta functions.
\end{remark}

\subsubsection{Mellin transforms}

{\lem \label{lemma_gamma_function_stuflöskdflkdjlkfjlksdfjlsd} Let $C$ be  a vertical contour such that $0<\Re(s)<\Re(d+1),\Re(d+r_1+1)$, for $x > 0$ we have,
\begin{equation}\label{eq:weirc_cos_integral}
\begin{split}
 \frac{1}{2 \pi i}  &\int_{C}\frac{   \Gamma(d-s+1)   \Gamma(s) \Gamma(r_1 + s) }{\Gamma (d+s+r_1+r_2+1)  }        \cdot   \cos(\pi s + \frac{\pi r_1}{2})   x^{-s}  ds 
 \\
  & = - \cos(\pi d + \pi r_1 / 2)  \left(\frac{1}{x}\right)^{1+d} \frac{\Gamma (d+1) \Gamma (d+r_1+1)}{  \Gamma(2 d+r_1+r_2+2) }  \, _2 F_1\left(d+1,d+r_1+1;2 d+r_1+r_2+2;\frac{1}{x}\right)\\
 & \ \ \ \ \ \ - \sin(\pi d + \pi r_1 / 2)\frac{ \pi  x^{-1-d}  (\frac{1}{x}-1)^{r_2} }{\Gamma(r_2+1)} \, _2F_1\left(d+r_2+1, k-d-1;r_2+1;1-\frac{1}{x}\right).
  \end{split}
\end{equation} }
\begin{proof} Since \begin{equation}
\begin{split}
\cos(\pi s + \pi r_1 / 2)  
&= \cos (\pi s - \pi d) \cos(\pi d + \pi r_1 / 2) -\sin (\pi s - \pi d) \sin(\pi d + \pi r_1 / 2),
\end{split}
\end{equation} we 
can rewrite \eqref{eq:weirc_cos_integral} as 
\begin{equation}\label{eq:integrals_from_inverse_Mellin_with_trigonometry}
\begin{split}
&\cos(\pi d + \pi r_1 / 2) \frac{1}{2 \pi i}  \int_{C} \frac{   \Gamma(d-s+1)   \Gamma(s) \Gamma(r_1 + s) }{\Gamma (d+s+r_1+r_2+1)  }        \cdot   \cos (\pi s - \pi d)   x^{-s}  ds \\
&  \ \ \ \ \  -  \sin(\pi d + \pi r_1 / 2) \frac{1}{2 \pi i}  \int_{C}\frac{   \Gamma(d-s+1)   \Gamma(s) \Gamma(r_1 + s) }{\Gamma (d+s+r_1+r_2+1)  }        \cdot   \sin (\pi s - \pi d)    x^{-s}  ds .
\end{split}
\end{equation}

For the first integral in \eqref{eq:integrals_from_inverse_Mellin_with_trigonometry}, we substitute  $w = -1+i\varepsilon$ and $w = -1-i\varepsilon$ into Lemma \ref{lem:mellin1}. 
Following the notation  \eqref{def:2f1_positionevee}, we obtain 
\begin{equation}
\begin{split}
     \, _2 F_1&\left(d+1,d+r_1+1;2 d+r_1+r_2+2;-\frac{-1+i\varepsilon}{t}\right) \\
    & \ \ \ \ \ \ \ \ \ \ \ \ +     \, _2 F_1\left(d+1,d+r_1+1;2 d+r_1+r_2+2;-\frac{-1-i\varepsilon}{t}\right) \\
    & \ \ \ \ \ \ \ \ \ \ \ \ \ \ \ \  \ \ \ \ \ \ \ \  \stackrel{ \varepsilon \to 0 }{\to} 2  \cdot   \, _2 F_1\left(d+1,d+r_1+1;2 d+r_1+r_2+2;\frac{1}{t}\right)
\end{split}
\end{equation}
and 
\begin{equation}\label{eq:2f1atvutis44}
\begin{split}
    (-1+i\varepsilon)^{s-d-1}+(-1-i\varepsilon)^{s-d-1} & 
    \stackrel{ \varepsilon \to 0 }{\to} 
    - 2 \cos(\pi ( s -  d) ), \\ 
\end{split}
\end{equation}
so we consider
\begin{equation}\label{eq:ssaadsdsdqq2213ll}
\begin{split}
- \frac{1}{4 \pi i}& \int_C \frac{\Gamma (d-s+1) \Gamma (s)  \Gamma (r_1+s)}{\Gamma (d+r_1+r_2+s+1)} \left((-1+i\varepsilon)^{s-d-1}+(-1-i\varepsilon)^{s-d-1} \right)x^{-s} ds\\
 & \to \frac{1}{2 \pi i} \int_C \frac{\Gamma (d-s+1) \Gamma (s)  \Gamma (r_1+s)}{\Gamma (d+r_1+r_2+s+1)} \cos( \pi s - \pi d   ) x^{-s} ds. 
    \end{split}
\end{equation}
By Lemma \ref{lem:mellin1},  for $0<\Re(s)<\Re(d+1),\Re(d+r_1+1)$,  the left hand side of \eqref{eq:ssaadsdsdqq2213ll} becomes 
\begin{equation}
\begin{split}
 &- x^{-d-1} \frac{\Gamma (d+1) \Gamma (d+r_1+1)}{ 2 \Gamma(2 d+r_1+r_2+2) }  \Big(\, _2 F_1\left(d+1,d+r_1+1;2 d+r_1+r_2+2;-\frac{-1+i\varepsilon}{x}\right)\\ &  \ \ \ \   \ \ \ \ \ \ \ \ \ \ \ \ \ \ \ \ + \,_2 F_1\left(d+1,d+r_1+1;2 d+r_1+r_2+2;-\frac{-1-i\varepsilon}{x}\right)\Big)
  \end{split}
\end{equation}
and taking the limit $\varepsilon\to 0$, we have 
\begin{equation}
\begin{split}
- \left(\frac{1}{x}\right)^{1+d} &\frac{\Gamma (d+1) \Gamma (d+r_1+1)}{  \Gamma(2 d+r_1+r_2+2) }  \, _2 F_1\left(d+1,d+r_1+1;2 d+r_1+r_2+2;\frac{1}{x}\right)\\
  &=  \frac{1}{2 \pi i} \int_C \frac{\Gamma (d-s+1) \Gamma (s)  \Gamma (r_1+s)}{\Gamma (d+r_1+r_2+s+1)} \cos( \pi s - \pi d   ) x^{-s} ds .
    \end{split}
\end{equation}

To study the second term in  \eqref{eq:integrals_from_inverse_Mellin_with_trigonometry}, we note 
\begin{equation}
\begin{split}
    (-1-i\varepsilon)^{s-d-1}-(-1+i\varepsilon)^{s-d-1} &\, \stackrel{ \varepsilon \to 0 }{\to}\,  2i \sin( \pi (s-d-1) )  = 2i \sin( \pi (s-d) ), \\ 
\end{split}
 \end{equation}  
then 
\begin{equation}\label{eq:löööllsslldlld3332}
\begin{split}
   -\frac{1}{4\pi }&  \int_{C}\frac{   \Gamma(d-s+1)   \Gamma(s) \Gamma(r_1 + s) }{\Gamma (d+s+r_1+r_2+1)  }        \cdot   \left((-1-i\varepsilon)^{s-d-1}-(-1+i\varepsilon)^{s-d-1}\right)   x^{-s}  ds 
    \\ & \stackrel{ \varepsilon \to 0 }{\to}
  \frac{1}{ 2\pi i}  \int_{C}\frac{   \Gamma(d-s+1)   \Gamma(s) \Gamma(r_1 + s) }{\Gamma (d+s+r_1+r_2+1)  }        \cdot   \sin (\pi s - \pi d)    x^{-s}  ds .
\end{split}
\end{equation}
Furthermore, by Lemma \ref{lem:mellin1}, for $0<\Re(s)<\Re(d+1),\Re(d+r_1+1)$,
we may rewrite the left hand side of \eqref{eq:löööllsslldlld3332} as 
\begin{equation}
\begin{split}
      &  \frac{1}{2i} x^{-1-d} \frac{\Gamma (d+1) \Gamma (d+r_1+1)}{  \Gamma(2 d+r_1+r_2+2) }   \big(\,_2 F_1\left(d+1,d+r_1+1;2 d+r_1+r_2+2;\frac{1+i\varepsilon}{x}\right)\\ & \ \ \ \ \ \ \ \ \ \ \ \ \ \ \ \ \ -\,_2 F_1\left(d+1,d+r_1+1;2 d+r_1+r_2+2;\frac{1-i\varepsilon}{x}\right)\big)\\
       & \to  x^{-1-d}     \frac{ \pi (\frac{1}{x}-1)^{r_2}  }{ \Gamma(r_2+1) } \, _2F_1\left(d+r_2+1, k-d-1;r_2+1;1-\frac{1}{x}\right)
\end{split}
\end{equation} as $\varepsilon\to 0$
since 
  \begin{equation}
    \begin{split}
  &     \frac{1}{\Gamma(k)}\left( \, _2 F_1 \left(d+1,d+r_1+1;k;\frac{1+i \varepsilon}{x}\right)  - \, _2 F_1\left(d+1,d+r_1+1;k;\frac{1-i \varepsilon}{x}\right)\right) \\
  &  \ \ \ \ \ \ \  \ \ \ \ \ \ \ \to  2  \pi i \frac{  (\frac{1}{x}-1)^{r_2} }{\Gamma(r_2+1)\Gamma (d+1) \Gamma (d+r_1+1)} \, _2F_1\left(d+r_2+1, k-d-1;r_2+1;1-\frac{1}{x}\right).
    \end{split}
    \end{equation}
    Thus 
    \begin{equation}
\begin{split}
  -  \sin(\pi d + \pi r_1 / 2)& \frac{1}{2 \pi i}  \int_{C}\frac{   \Gamma(d-s+1)   \Gamma(s) \Gamma(r_1 + s) }{\Gamma (d+s+r_1+r_2+1)  }        \cdot   \sin (\pi s - \pi d)    x^{-s}  ds\\
   & = -     \frac{ \pi \sin(\pi d + \pi r_1 / 2) x^{-1-d}  (\frac{1}{x}-1)^{r_2} }{\Gamma(r_2+1)} \, _2F_1\left(d+r_2+1, k-d-1;r_2+1;1-\frac{1}{x}\right),
\end{split}
\end{equation}
that finishes the proof.
\end{proof}

\subsection{Toward Theorem \ref{thm:THEOREM-A-precursor}: Contribution from Petersson trace formula}\label{sec:petersson}

Let $n>0$ and let $r_1, r_2, d \in \mathbb{C}$ satisfy \eqref{eq:cond_1} and \eqref{cond_2a}, let $\kappa(m) = \sigma_{-r_1}(m)$ and let 
\begin{equation}\label{eq:definition_R}
	-R = \frac{1}{2} (r_1-r_2-1),
\end{equation}
so that \eqref{eq:integral_after_contour_changedfsdd4} becomes
\begin{equation}\label{eq:sum_conv_defined}
        \mathcalSsumdef= 	 \sum_{m=1}^\infty \sigma_{-r_1}(m) m^{-R}  \sum_{\ell=1}^\infty \frac{2 \pi }{\ell} S(-m,-n; \ell)  J_{2d+r_1+r_2+1} ( \tfrac{4 \pi \sqrt{mn}}{\ell}).
\end{equation}


We now express $\mathcalSsumdef$ in terms of $L$-values. Substituting the Petersson Trace Formula (Lemma \ref{lem:petersson_trace_formula}) into the definition of $\mathcalSsumdef$ as in \eqref{eq:integral_after_contour_changedfsdd4},  we obtain for $\Re(r_2-r_1) > 2$,
\begin{equation}
\begin{split}\label{eq:Spetersson1}
        \mathcalSsumdef = &  \sum_{m=1}^\infty \sigma_{-r_1}(m) m^{-R} (-1)^{d+1+\tfrac{r_1+r_2}{2}} \frac{\Gamma(k-1)}{ (4 \pi \sqrt{mn})^{2d+r_1+r_2+1} } \sum_{f \in \mathcal{F}} \frac{a_f(n) \overline{a_f(m)}}{\langle \bar f, \bar f \rangle }   \\ 
         & \ \ \ \ \ \ \ \ \ \ \ \ \ \ \ \ \ \ - \sum_{m=1}^\infty \sigma_{-r_1}(m) m^{-R} (-1)^{d+1+\tfrac{r_1+r_2}{2}} \delta_{m=n}.
\end{split}    
\end{equation}

Note that \eqref{eq:divisor_estimate} and \eqref{eq:J_ineq} together with a trivial estimate on Kloosterman sums $|S(-m,-n;\ell)| \le \ell$  imply that there exists $\varepsilon>0$ such that every summand in the right hand side of \eqref{eq:sum_conv_defined} can be evaluated from above by
\begin{equation}\label{eq:convergence_discussion_absdlslskdkflk}    (\frac{\sqrt{m}}{\ell})^{ \Re (r_2-r_1)-1-\varepsilon} \cdot m^{ \Re(   \frac{1}{2} (r_1-r_2-1)) } = \ell^{-\Re(r_2-r_1)+1+\varepsilon} m^{-1 - \varepsilon/2}.\end{equation}
Thus the sum in the left hand side of \eqref{eq:Spetersson1} absolutely converges for $\Re(r_2-r_1) > 2$.

By \eqref{eq:Hecke_sum_product_of_Lvalues}, for $\Re(r_2-r_1) > 2$, the contribution from the first term in the right hand side of  \eqref{eq:Spetersson1} is
\begin{equation}
\begin{split}\label{eq:Sterm1}
&  \sum_{m=1}^\infty  \sigma_{-r_1}(m)   m^{-R} (-1)^{d+1+\tfrac{r_1+r_2}{2}} \frac{\Gamma(k-1)}{ (4 \pi \sqrt{mn})^{2d+r_1+r_2+1} } \sum_{f \in \mathcal{F}} \frac{a_f(n)  \overline{a_f(m)} }{\langle \bar f, \bar f \rangle}    \\
  & = (-1)^{d+1+\tfrac{r_1+r_2}{2}} \frac{\Gamma(k-1)}{ (4 \pi \sqrt{n})^{2d+r_1+r_2+1}} \sum_{f \in \mathcal{F}}  \frac{L( d+r_2+1,\bar f ) L( d+r_2+r_1+1,\bar f ) }{\zeta(2 ( d+r_2+1 - \frac{k-1}{2}  ) + r_1 )} \frac{a_f(n)}{\langle \bar f, \bar f \rangle}.
\end{split} 
\end{equation}
The contribution from the second term  in the right hand side of  \eqref{eq:Spetersson1} is
\begin{equation}
 - (-1)^{d + 1 + \frac{r_1+r_2}{2}}   \sum_{m=1}^\infty \sigma_{-r_1}(m) m^{-R}  \delta_{m=n} =    (-1)^{d  + \frac{r_1+r_2}{2}}  \sigma_{-r_1}(n) n^{\frac{1}{2} (r_1-r_2-1) }.
\end{equation}
We obtain the following lemma:
\begin{lemma}\label{lem:side1}
Let $\Re(r_2)>\Re(r_1)+2$, 
 and let $ \mathcal{F}$ be a basis of normalized Hecke eigenforms of weight  $k = 2d+r_1+r_2+2$ is an even integer with $k \ge 4$, then for $n \in \mathbb{N}$,
\begin{align}\mathcalSsumdef&=  \sum_{m=1}^\infty \sigma_{-r_1}(m) m^{-\frac{1}{2} (r_1-r_2-1)}  \sum_{\ell=1}^\infty \frac{2 \pi }{\ell} S(-m,-n; \ell)  J_{2d+r_1+r_2+1} ( \tfrac{4 \pi \sqrt{mn}}{\ell})\nonumber\\ & =   \ (-1)^{d+1+\tfrac{r_1+r_2}{2}} \frac{\Gamma(k-1)}{ (4 \pi \sqrt{n})^{2d+r_1+r_2+1}} \sum_{f \in \mathcal{F}}  \frac{L( d+r_2+1,\bar f ) L( d+r_2+r_1+1,\bar f ) }{\zeta(2 ( d+r_2+1 - \frac{k-1}{2}  ) + r_1 )} a_f(n) / \langle \bar f, \bar f \rangle  \nonumber \\ & \ \ \ \ \ \ \  +   (-1)^{d  + \frac{r_1+r_2}{2}}  \sigma_{-r_1}(n) n^{\frac{1}{2} (r_1-r_2-1) } .\end{align}    
\end{lemma}

\subsection{Toward Theorem \ref{thm:THEOREM-A-precursor}: Contribution from the inverse Mellin transform}\label{sec:InvMellin}

Let 
\begin{equation}
    c = - \Re r_1 - \varepsilon
\end{equation}
for some sufficiently small $\varepsilon>0$. Then from \eqref{cond_2a} and \eqref{eq:cond_1}, the constant $c$ satisfies 
\begin{equation}\label{eq:recasdaskkkk}
   \Re ( \frac{1-r_1-r_2}{2} )  < c < \Re (d) +1, 
\end{equation}
and we  substitute \eqref{eq:gamma_mellin_inverstion}  into  \eqref{eq:sum_conv_defined} 
to obtain for $\text{Re}(r_2-r_1)>2$,  
\begin{equation}\label{eq:kappa_skjdfkjsdflgjsfg}
	\begin{split}
	\mathcalSsumdef =& \sum_{m=1}^\infty \sigma_{-r_1}(m) m^{-R}  \sum_{\ell=1}^\infty \frac{2 \pi }{\ell} S(-m,-n; \ell)  J_{2d+r_1+r_2+1} ( \tfrac{4 \pi \sqrt{mn}}{\ell})\\
    =& \sum_{m=1}^\infty \sigma_{-r_1}(m) m^{-R}  \sum_{\ell=1}^\infty \frac{2 \pi }{\ell} S(-m,-n; \ell) 	\frac{1}{2 \pi i} \int_{c-i \infty}^{c+i \infty}   \frac{   \Gamma(d-s+1)  }{\Gamma (d+s+r_1+r_2+1)  }      \left( \frac{2 \pi \sqrt{mn}}{\ell} \right)^{2s+r_1+r_2-1} ds. 
	\end{split}
\end{equation}

On the right hand side of \eqref{eq:kappa_skjdfkjsdflgjsfg}, we would like to swap the order of  summation over $\ell$, summation over~$m$ and integration. In order to justify this we start by noting by \cite[5.11.9]{NIST}
 for $s = c + i \tau$, \begin{equation}\label{eq:fractions_of_gamma_slskldkjfjjfwew}
    \begin{split}
         \left| \frac{   \Gamma(d-s+1)  }{\Gamma (d+s+r_1+r_2+1)  } \right|   \sim \frac{|\tau|^{\Re d-c +1}}{|\tau|^{\Re (d+r_1+r_2)+c+1}} \sim |\tau|^{-2c- \Re r_1- \Re r_2} \stackrel{\eqref{eq:recasdaskkkk}}{\le} |\tau|^{-1-2 \varepsilon}.
    \end{split}
\end{equation}
Thus the integral in the right hand side of \eqref{eq:kappa_skjdfkjsdflgjsfg} can be evaluated as 
\begin{equation}
\begin{split}
 \int_{c-i \infty}^{c+i \infty}& \left|  \frac{   \Gamma(d-s+1)  }{\Gamma (d+s+r_1+r_2+1)  }      \left( \frac{2 \pi \sqrt{mn}}{\ell} \right)^{2s+r_1+r_2-1} \right| ds  
\lesssim\left( \frac{\sqrt{mn}}{\ell} \right)^{2 c+\Re r_1+\Re 
r_2-1}
\end{split}
\end{equation}
and the sum is absolutely convergent.

Changing the order of summation in $m, \ell$, integration and substituting the definition of the Estermann zeta function and using \eqref{eq:gammaeps}, we obtain that  \eqref{eq:kappa_skjdfkjsdflgjsfg} becomes
\begin{align}
	\mathcalSsumdef &= \frac{1}{2 \pi i} \int_{c-i \infty}^{c+i \infty}  \sum_{\ell=1}^\infty \frac{2 \pi }{\ell}   \sumk e(\tfrac{-\kestermann n}{\ell}) \frac{   \Gamma(d-s+1)  }{\Gamma (d+s+r_1+r_2+1)  }   \\
    & \ \ \ \ \ \ \ \ \ \ \ \ \ \ \ \ \ \ \ \ \ \ \ \ \ \ \ \ \ \ \ \ \ \ \ \ \ \ \times \left( \frac{2 \pi \sqrt{n}}{\ell} \right)^{2s+r_1+r_2-1} E(1-s-r_1,-\bar{\kestermann}/\ell,-r_1)  ds\nonumber \\  
    &=\frac{1}{2 \pi i}   \sum_{\ell=1}^\infty \frac{2 \pi }{\ell}   \sumk e(\tfrac{-\kestermann n}{\ell}) \int_{c-i \infty}^{c+i \infty}\mathcal{E}(s) ds.
\end{align}

We now move the contour from the vertical line $\Re(s) = c$ to a vertical line satisfying $1< \Re(s) <\text{Re}(d)+1$.
By Lemma \ref{lem:contour}, for $s=\sigma+i\tau$ with fixed $\sigma$ and all $\delta>0$, 
\begin{equation}\label{eq:asymp_Omega1}
\begin{split}
    |\mathcal{E}(s) |   \lesssim |\tau|^{{\text{Re}(r_1-r_2)+1+2\delta}},
\end{split}
\end{equation}
for $ \tau \to \pm \infty$. Thus for $-\text{Re}(r_1)-\varepsilon < \sigma < 1 + \varepsilon$,
if $\text{Re}(r_2)>\text{Re}(r_1)+1$, then $|\tau|^{\text{Re}(r_1-r_2)+1+2\delta} \to 0$ for $|\tau|\to \infty$ and
\newcommand{\Rbig}{\tau'}
\begin{equation}\label{eq:horizontal_contour_vanishes}
    \int_{-r_1-\varepsilon + \Rbig i}^{1+\varepsilon + \Rbig i} \mathcal{E}(s) \, ds  +     \int_{-r_1-\varepsilon - \Rbig i}^{1+\varepsilon - \Rbig i} \mathcal{E}(s) \, ds  \stackrel{\Rbig \to \infty}{\to} 0.
\end{equation}
Thus we can shift of the contour to the right and, using Section \ref{sec:poles_of_E_sdklfjfjfskdkdkkk3}, we pick up the residues of two poles to get
\begin{equation}\label{eq:intermediate_integral_petersson_backwards_sdfsdfsdeee3324}
	\begin{split}
	\mathcalSsumdef   
    &=\frac{1}{2 \pi i}   \sum_{\ell=1}^\infty \frac{2 \pi }{\ell}   \sumk e(\tfrac{-\kestermann n}{\ell}) \int_{\tilde c-i \infty}^{\tilde c+i \infty}\mathcal{E}(s)\, ds \\
    & +  \sum_{\ell=1}^\infty \frac{2 \pi }{\ell}    \sumk e(\tfrac{-\kestermann n}{\ell})\Big( \text{Res}_{s=0}\,\mathcal{E}(s)+ \text{Res}_{s=-r_1}\,\mathcal{E}(s)\Big)
	\end{split}
\end{equation}
for some $\tilde c \in \mathbb{R}$ satisfying $1<\tilde c <\text{Re}(d)+1$ where the residues are given in \eqref{eq:pole0} and \eqref{eq:pole-r1}.

The residue at $s=0$ gives rise to the term
\begin{equation}
\begin{split}
\sum_{\ell=1}^\infty \frac{2 \pi }{\ell}   \sumk e(\tfrac{-\kestermann n}{\ell}) \text{Res}_{s=0}\,\mathcal{E}(s)  
&=-\frac{ 2 \pi   \Gamma(d+1)  \left(2 \pi \sqrt{n} \right)^{r_1+r_2-1} }{\Gamma (d+r_1+r_2+1)  } \zeta(1-r_1) \sum_{\ell=1}^\infty     \sumk e(\tfrac{-\kestermann n}{\ell})    \ell^{-r_2-1}\\
& \stackrel{\eqref{eq:sigma0}}{=} -\frac{ 2 \pi   \Gamma(d+1)  }{\Gamma (d+r_1+r_2+1)  }  \left(2 \pi \sqrt{n} \right)^{r_1+r_2-1} \frac{\zeta(1-r_1)}{\zeta(1+r_2)}\sigma_{-r_2}(n)
\end{split}    
\end{equation}
and the residue at $s=-r_1$ yields
\begin{equation}
\begin{split}
\sum_{\ell=1}^\infty \frac{2 \pi }{\ell}  & \sumk e(\tfrac{-\kestermann n}{\ell})  \text{Res}_{s=-r_1}\,\mathcal{E}(s) \\
& = -\frac{ 2 \pi  \Gamma(d+r_1+1)  }{\Gamma (d+r_2+1)  }      \left( 2 \pi \sqrt{n}\right)^{-r_1+r_2-1}\zeta(1+r_1)\sum_{\ell=1}^\infty   \sumk e(\tfrac{-\kestermann n}{\ell})     \ell^{-r_2-1}\\ 
& \stackrel{\eqref{eq:sigma0}}{=} - \frac{ 2 \pi  \Gamma(d+r_1+1)  }{\Gamma (d+r_2+1)  }      \left( 2 \pi \sqrt{n}\right)^{-r_1+r_2-1}\frac{\zeta(1+r_1)}{\zeta(1+r_2)}\sigma_{-r_2}(n).    
\end{split}    
\end{equation}
For the remaining integral in \eqref{eq:intermediate_integral_petersson_backwards_sdfsdfsdeee3324}, we apply the functional equation to $\mathcal{E}(s)$ as in \eqref{eq:E_funeq}, the integral becomes 
\begin{equation}\label{eq:t1t2}
	\begin{split}
&\frac{1}{2 \pi i}   \sum_{\ell=1}^\infty \frac{2 \pi }{\ell}   \sumk e(\tfrac{-\kestermann n}{\ell}) \int_{\tilde c-i \infty}^{\tilde c+i \infty}\mathcal{E}(s)\, ds \\
          &   = 		  \sum_{\ell=1}^\infty   \left( \frac{2 \pi }{\ell} \right)^{r_2+1}   \sumk e(\tfrac{-\kestermann n}{\ell}) \left(  \sqrt{n}  \right)^{r_1+r_2-1} \times \frac{1}{2 \pi i}  \int_{\tilde c-i \infty}^{\tilde c+i \infty} \frac{   \Gamma(d-s+1)   \Gamma(s)  }{\Gamma (d+s+r_1+r_2+1) \Gamma(1-r_1-s)  }\\
&      \ \ \ \ \ \ \ \ \  \ \ \ \ \ \  \ \ \ \ \ \  \times       \left(  \frac{\cos(\pi s + \frac{\pi r_1}{2}) E(s;\kestermann/\ell,-r_1) + \cos(- \pi r_1/2 ) E(s;-\kestermann/\ell,-r_1)  }{\sin(\pi (r_1+s))} \right) n^{s}  ds.
    \end{split}
    \end{equation}
We split \eqref{eq:t1t2} into a sum $\mathcal{T}_1 + \mathcal{T}_2$, 
where 
\begin{equation}\label{def:t1}
\begin{split}
\mathcal{T}_1 &:=  \sum_{\ell=1}^\infty   \left( \frac{2 \pi }{\ell} \right)^{r_2+1}   \sumk e(\tfrac{-\kestermann n}{\ell}) \left(  \sqrt{n}  \right)^{r_1+r_2-1}  \\
&      \ \ \ \ \ \ \ \ \  \ \ \ \ \ \  \ \ \  \times  \frac{1}{2 \pi i}\int_{\tilde c-i \infty}^{\tilde c+i \infty} \frac{   \Gamma(d-s+1)   \Gamma(s)  }{\Gamma (d+s+r_1+r_2+1) \Gamma(1-r_1-s)  }          \frac{\cos(\pi s + \frac{\pi r_1}{2}) E(s;\kestermann/\ell,-r_1)  }{\sin(\pi (r_1+s))}  n^{s}  ds 
\end{split}    
\end{equation}
and 
\begin{equation}
\begin{split}\label{def:t2}
    \mathcal{T}_2 &:=  \sum_{\ell=1}^\infty   \left( \frac{2 \pi }{\ell} \right)^{r_2+1}   \sumk e(\tfrac{-\kestermann n}{\ell}) \left(  \sqrt{n}  \right)^{r_1+r_2-1}  \\
&      \ \ \ \ \ \ \ \ \  \ \ \ \ \ \   \ \ \  \times      \frac{1}{2 \pi i}  \int_{\tilde c-i \infty}^{\tilde c+i \infty} \frac{   \Gamma(d-s+1)   \Gamma(s)  }{\Gamma (d+s+r_1+r_2+1) \Gamma(1-r_1-s)  }  \ \frac{ \cos(- \pi r_1/2 ) E(s;-\kestermann/\ell,-r_1)  }{\sin(\pi (r_1+s))}  n^{s}  ds 
\end{split}
\end{equation}
and address each term in the following sections.

\subsubsection{$\mathcal{T}_1$}\label{sec:T1}

After we have moved the contour to the right, $E(s;\kestermann/\ell,-r_1)$ and $E(s;-\kestermann/\ell,-r_1)$ are now in the domain of absolute convergence. Thus we rewrite the Estermann zeta functions as an infinite sum over $n_1 \in \mathbb{N}_1$ using the definition
	\begin{equation}
		E(s; \kestermann/\ell, -r_1) = \sum_{n_1=1}^\infty \sigma_{-r_1}(n_1) e( \tfrac{\kestermann n_1}{\ell}) n_1^{-s}\quad  \text{and} \quad E(s; -\kestermann/\ell, -r_1) = \sum_{n_1=1}^\infty \sigma_{-r_1}(n_1) e( -\tfrac{\kestermann n_1}{\ell}) n_1^{-s},
	\end{equation}
    for $\Re(s)>1.$ 
The first term $\mathcal{T}_1$ in \eqref{def:t1} becomes
\begin{equation}\label{eq:intermediate_identity_in_petersson_backwardsasdssd2231244}
    \begin{split}
     \mathcal{T}_1=   
   &   \sum_{n_1=1}^\infty  		  \sum_{\ell=1}^\infty   \left( \frac{2 \pi }{\ell} \right)^{r_2+1}   \sumk e(\tfrac{\kestermann (n_1- n)}{\ell}) \sigma_{-r_1}(n_1) \,   n^{  \frac{ r_1+r_2-1 }{2}   } \times  \\
&   \ \ \   \ \ \ \times \frac{1}{2 \pi i}  \int_{\tilde c-i \infty}^{\tilde c+i \infty} \frac{   \Gamma(d-s+1)   \Gamma(s)  }{\Gamma (d+s+r_1+r_2+1) \Gamma(1-r_1-s)  }        \cdot   \frac{\cos(\pi s + \frac{\pi r_1}{2})   }{\sin(\pi (r_1+s))} n^{s} n_1^{-s}  ds
    \end{split}
\end{equation}
when $1<\tilde c<\Re(d)+1$. 
Using the functional equation for the Gamma function, we obtain that the integral above is equal to
\begin{equation}\label{eq:Upsilon_first_appearence}
\begin{split}
\mathcal{I}_1(n_1,n) & :=\frac{1}{2 \pi i}  \int_{\tilde c-i \infty}^{\tilde c+i \infty} \frac{   \Gamma(d-s+1)   \Gamma(s)  }{\Gamma (d+s+r_1+r_2+1) \Gamma(1-r_1-s)  }        \cdot   \frac{\cos(\pi s + \frac{\pi r_1}{2})   }{\sin(\pi (r_1+s))} n^{s} n_1^{-s}  ds \\
& = \frac{1}{\pi} \frac{1}{2 \pi i}  \int_{\tilde c-i \infty}^{\tilde c+i \infty} \frac{   \Gamma(d-s+1)   \Gamma(s) \Gamma(r_1 + s) }{\Gamma (d+s+r_1+r_2+1)  }        \cdot   \cos(\pi s + \frac{\pi r_1}{2})   n^{s} n_1^{-s}  ds.
\end{split}
\end{equation}

By Lemma \ref{lemma_gamma_function_stuflöskdflkdjlkfjlksdfjlsd}, given that $\Re(r_1) > 0$,
we get that $\mathcal{I}_1(n_1,n)$ equals 
\begin{equation}\label{eq:upsilon}
    \begin{split}
  & - \frac{1}{\pi} \cos(\pi d + \pi r_1 / 2) (\frac{n}{n_1})^{1+d} \frac{\Gamma (d+1) \Gamma (d+r_1+1)}{  \Gamma(2 d+r_1+r_2+2) }  \, _2 F_1\left(d+1,d+r_1+1;2 d+r_1+r_2+2;\frac{n}{n_1}\right) \\
& - \delta_{(n_1 \le n)}   \sin(\pi d + \pi r_1 / 2) \frac{(\tfrac{n}{n_1})^{d+1} (\tfrac{n}{n_1}-1)^{r_2} }{\Gamma (r_2+1)} \, _2F_1 \left(d+r_2+1,d+r_1+r_2+1;r_2+1;1-\frac{n}{n_1} \right).
    \end{split}
\end{equation}

When $n>n_1$, the first hypergeometric function is understood as a half-sum of hypergeometric functions above and below the split; cf. \eqref{def:2f1_positionevee}. Using \eqref{eq:ramanujan_identity}, we obtain  that \eqref{eq:intermediate_identity_in_petersson_backwardsasdssd2231244} becomes 
\begin{equation}\label{eq:sum_with_negative_divisors}
    \begin{split}
  \mathcal{T}_1 & = \sum_{n_1=1}^\infty  		  \sum_{\ell=1}^\infty     \left( \frac{2 \pi }{\ell} \right)^{r_2+1}   \sumk e(\tfrac{ \kestermann (n_1 - n) }{\ell}) \sigma_{-r_1}(n_1) \,  n^{  \frac{ r_1+r_2-1 }{2}   } \mathcal{I}_1(n_1,n) \\ 
& = \frac{(2 \pi)^{r_2+1}}{ \zeta(1+r_2) }   \sum_{n_1=1}^\infty  		   \sigma_{-r_1}(n_1) \sigma_{-r_2}(|n_1-n|)  n^{  \frac{ r_1+r_2-1 }{2}   } \mathcal{I}_1(n_1,n).      \\
    \end{split}
\end{equation}

Another boundary term corresponds to the term $n=n_1$ and, by  \eqref{eq:sigma0}, gives
\begin{equation}\label{eq:yet_anothersdfsldkjfslkdjfsdknk}
\begin{split}
 \frac{(2 \pi)^{r_2+1}}{ \zeta(1+r_2) }     		   \sigma_{-r_1}(n) \sigma_{-r_2}(0)  n^{  \frac{ r_1+r_2-1 }{2}   } \mathcal{I}_1(n,n) 
=   \frac{(2 \pi)^{r_2+1}}{ \zeta(1+r_2) }     		   \sigma_{-r_1}(n) \zeta(r_2)  n^{  \frac{ r_1+r_2-1 }{2}   } \mathcal{I}_1(n,n).
\end{split}
\end{equation}
Given that 
\begin{equation}
    \mathcal{I}_1(n,n) = -\frac{\Gamma (d+1) \Gamma (r_2) \cos \left(\pi  d+\frac{\pi  r_1}{2}\right) \Gamma (d+r_1+1)}{\pi  \Gamma (d+r_2+1) \Gamma (d+r_1+r_2+1)},
\end{equation}
this term \eqref{eq:yet_anothersdfsldkjfslkdjfsdknk} becomes 
\begin{equation}
    -\frac{\Gamma (d+1) \Gamma (r_2) \cos \left(\pi  d+\frac{\pi  r_1}{2}\right) \Gamma (d+r_1+1)}{\pi  \Gamma (d+r_2+1) \Gamma (d+r_1+r_2+1)} \frac{(2 \pi)^{r_2+1}}{ \zeta(1+r_2) }     		   \sigma_{-r_1}(n) \zeta(r_2)  n^{  \frac{ r_1+r_2-1 }{2}   }.
\end{equation}

\subsubsection{$\mathcal{T}_2$}\label{sec:T2}
Similarly to $\mathcal{T}_1$,  the second term $\mathcal{T}_2$ in \eqref{def:t2} becomes
\begin{equation}\label{eq:intermediate_identity_in_petersson_backwardsasdssd22312442}
    \begin{split}
      \mathcal{T}_2 
         &  = \sum_{n_1=1}^\infty \sum_{\ell=1}^\infty   \left( \frac{2 \pi }{\ell} \right)^{r_2+1}   \sumk \sigma_{-r_1}(n_1) e(-\tfrac{\kestermann n_1}{\ell}) e(\tfrac{-\kestermann n}{\ell}) n^{ \frac{r_1+r_2-1}{2}} \times  \\
&      \ \ \   \ \ \    \times \frac{1}{2 \pi i}  \int_{\tilde c-i \infty}^{\tilde c+i \infty} \frac{   \Gamma(d-s+1)   \Gamma(s)  }{\Gamma (d+s+r_1+r_2+1) \Gamma(1-r_1-s)  }        \cdot   \frac{ \cos(- \pi r_1/2 )  n_1^{-s}   }{\sin(\pi (r_1+s))} n^{s}  ds .
    \end{split}
\end{equation}
By Lemma \ref{lem:mellin1}, 
\begin{equation}\label{eq:Xi}
    \begin{split}
&   \mathcal{I}_2(n_1, n)   := \cos( \pi r_1/2 )  \times   \frac{1}{2 \pi i}  \int_{\tilde c-i \infty}^{\tilde c+i \infty} \frac{   \Gamma(d-s+1)   \Gamma(s)  }{\Gamma (d+s+r_1+r_2+1) \Gamma(1-r_1-s)  }      \csc(\pi (r_1+s)) n^{s} n_1^{-s}   ds  \\
 & \ \ \ \ \  = \frac{(n_1 / n)^{-d-1}}{\pi} \cos( \pi r_1/2 ) \frac{\Gamma (d+1)  \Gamma (d+r_1+1)}{ \Gamma(2 d+r_1+r_2+2) }    \, _2F _1  \left(d+1,d+r_1+1;2 d+r_1+r_2+2;-\frac{n}{n_1}\right),
    \end{split}
\end{equation}
and we have 
\begin{equation}
    \begin{split}
      \mathcal{T}_2    & =  \sum_{n_1=1}^\infty \sum_{\ell=1}^\infty   \left( \frac{2 \pi }{\ell} \right)^{r_2+1}   \sumk \sigma_{-r_1}(n_1) e(-\tfrac{\kestermann n_1}{\ell}) e(\tfrac{-k n}{\ell}) n^{ \frac{r_1+r_2-1}{2}} \mathcal{I}_2(n_1, n)\\
 &  = \sum_{n_1=1}^\infty  \sigma_{-r_1}(n_1) \mathcal{I}_2(n_1,n) n^{ \frac{r_1+r_2-1}{2}}  \sum_{\ell=1}^\infty   \left( \frac{2 \pi }{\ell} \right)^{r_2+1}   \sumk  e(-\tfrac{\kestermann (n_1+n)}{\ell})   \\
    & \stackrel{\eqref{eq:ramanujan_identity}}{=}  \frac{(2 \pi)^{r_2+1}}{  \zeta(1+r_2) } \sum_{n_1=1}^\infty  \sigma_{-r_1}(n_1) \sigma_{-r_2}(| n_1+n|) \mathcal{I}_2(n_1,n) n^{ \frac{r_1+r_2-1}{2}}. 
    \end{split}
\end{equation}

\subsection{Conclusion of proof of Theorem \ref{thm:THEOREM-A-precursor}}\label{sec:conclusion}

Setting the equations in Lemma \ref{lem:side1} and \eqref{eq:intermediate_integral_petersson_backwards_sdfsdfsdeee3324} equal and applying the formulas Sections \ref{sec:T1} and \ref{sec:T2}, we arrive to the statement of Theorem~\ref{thm:THEOREM-A-precursor}.

\subsection{Alternative formulation of the proof of Theorem \ref{thm:THEOREM-A-precursor}}\label{sec:sketches}
As mentioned above, the proof of Theorem~\ref{thm:THEOREM-A-precursor} given above addresses sums $\mathcalSsumdef$ as in \eqref{eq:integral_after_contour_changedfsdd4}  with $\kappa(m) =\sigma_{-r_1}(m)$ in two different ways. On one hand,  we use the Petersson trace formula to derive the contribution of the Hecke eigenform and on the other hand we use a Mellin inversion identity for the $J$-Bessel function to derive the convolution sum. However, once one identifies the function $\mathcal{Q}$ desired,  there is a more constructive version of this proof outlined as follows.

Following Motohashi \cite{motohashi1994binary}, we rewrite the sum 
\begin{equation}\label{eq:Qsum}\sum_{n_1 \in \mathbb{Z} \setminus \{ 0, n \} }  \mathcal{Q}(n_1,n-n_1)  \sigma_{-r_1}(n_1) \sigma_{-r_2}(n-n_1) \end{equation}
using Ramanujan's identity  
$\displaystyle
		\sigma_{-r_2}( | x | )   = 
        \zeta(1+r_2)   \sum_{\ell = 1}^\infty \ell^{-r_2 - 1} c_\ell (x)$
for  $\text{Re}(r_2) >  0$  and  $x \in \mathbb{Z} \setminus \{0 \}$  	
as in \cite[(2)]{murty2013ramanujan}. One of the four ``boundary terms" (from $Z^{(\alpha,\beta)}_d$ as in \eqref{def:Zabd}) comes from when $n_1=n$ and so \eqref{eq:Qsum} can be written 
\begin{equation}\label{eq:addbound}
\sum_{n_1=1}^\infty \sigma_{-r_1}(n_1) \left[ e(\tfrac{v n_1}{\ell}) \Psi(n_1)  + e(-\tfrac{v n_1}{\ell}) \widetilde\Psi(n_1) \right] - \sigma_{-r_1}(n) e(\tfrac{vn }{\ell})  \Psi(n)
\end{equation}
for some function $\Psi$ where  $\widetilde\Psi(y) = \Psi(-y) $. We then rewrite $\Psi$ using Mellin inversion with  $1 < c_1 < \text{Re}(d)+ 1$ so that the first two terms of \eqref{eq:addbound} become 
\begin{equation}
 \frac{1}{2 \pi i } \int_{c_1-i \infty}^{c_1+i\infty}  \mathcal{M}  \Psi(s) E(s,v/\ell,-r_1)\, ds 
            + \frac{1}{2 \pi i }	\int_{c_1-i \infty}^{c_1+i\infty} \mathcal{M} \widetilde\Psi(s) E(s,-v/\ell,-r_1)\, ds ,
\end{equation}
where $E(s,v/\ell,-r_1)$ is the Estermann zeta function. Next one shifts the contour of integration just to the left of $-\Re(r_1)$.

Two of the three remaining boundary terms appear as residues for poles at $s=0$ and $-r_1$. The remaining integral can be written as 
\begin{equation}\label{eq:integral_after_contour_change}
           \begin{split}
    \frac{  (-1)^{r_1/2+r_2/2}  }  {2}{d+r_2 \choose d} &\zeta'(-r_2) n^{\frac{1}{2}(r_1+r_2+1)}  \sum_{\ell=1}^\infty \sum_{m=1}^\infty \frac{1}{\ell} S(-m,-n; \ell)    \sigma_{r_1}(m) m^{-\frac{1}{2} (r_1+r_2+1)}   \\
           & \times \frac{2\pi}{2 \pi i } \left( \int_{c_2-i \infty}^{c_2+i \infty}   \frac{\Gamma(1+d-s)}{\Gamma(r_1+r_2+d+s+1)}  \left( \frac{2 \pi \sqrt{mn}}{\ell} \right)^{2s+r_1+r_2-1} ds   \right) ,
\end{split}
\end{equation}
where $c_2=-\text{Re}(r_1)-\epsilon$. The integral in $c_2$  gives rise to the $J$-Bessel function $J_{2d+r_1+r_2+1} ( \tfrac{4 \pi \sqrt{mn}}{\ell})$. It remains to apply the Petersson trace formula so that the last boundary term arises when $m=n$. We now note that a different choice of $\mathcal{Q}$ will not result in the appropriate $J$-Bessel function and other terms will arise form the Kuznetsov trace formula.

\subsection{Meromorphic continuation (proof of Theorem \ref{thm:THEOREM-A})}\label{sec:mero}

In this section, we derive Theorem~\ref{thm:THEOREM-A} from Theorem~\ref{thm:THEOREM-A-precursor}. 
The sketch of the proof is as follows. 
We first show with the help of Morera's theorem that the left-hand side of the identity~\eqref{eq:first_main_thm_statement_sdfsdffff}   defines an analytic function in the variables $r_1$, $r_2$, $d$ within the domain defined by the  condition~\eqref{eq:cond_1}. Then we show that the  right-hand side of~\eqref{eq:first_main_thm_statement_sdfsdffff} defines a meromorphic function in all three variables in the complex plane. 
After that, the analytic continuation will imply that the identity~\eqref{eq:first_main_thm_statement_sdfsdffff} extends to a larger domain than initially established in Theorem~\ref{thm:THEOREM-A-precursor} -- more precisely, to parameters satisfying~\eqref{eq:cond_1} subject to the restriction that $k = 2d+2+r_1+r_2 \in 2 \mathbb{N} + 4$. 

To proceed with the proof, we first show that the left side of \eqref{eq:first_main_thm_statement_sdfsdffff} converges absolutely and uniformly for $r_1, r_2, d$ belonging to compact subsets $K$ of $\mathbb{R}^3$ avoiding non-positive integers $-\mathbb{N}_0$ in the third coordinate.
For small values of $|n/n_1|$, the hypergeometric function $\, _2 F_1 (d+1,d+r_1+1;k;\frac{n}{n_1})$ is defined using the Gauss series \cite[15.2.1]{NIST}. Moreover, for $(a,b,c) \in K$, 
\begin{equation}\label{eq:ssoksdosdodofkfkfkfkslölsls}
\left| \frac{(a)_m (b)_m}{(c)_m m!} \right|   \le_K \left| \frac{\Gamma(a+m)}{\Gamma(c+m)}  \frac{\Gamma(b+m)}{\Gamma(m)} \right|.
\end{equation}
We want to uniformly\footnote{Uniformly in parameters $a, b, c$ belonging to a compact set.} bound the right hand side of \eqref{eq:ssoksdosdodofkfkfkfkslölsls} as a function of $m$. Recall that $c$ belongs to a compact set. Then, by the recurrence relation of the Gamma function,  there is a constant $M_1 \in \mathbb{R}$ and $c_1 \in \mathbb{C}$ with $\Re c_1 \gg 0$ such that $|\Gamma(c+m)| \ge |\Gamma(c_1+m)| m^{M_1} $.  Since the imaginary part of $c_1$ (and hence, $c_1+m$) is uniformly bounded on compacts, by \cite[5.6.7]{NIST}, $|\Gamma(c_1+m)| \ge_K |\Gamma(\Re c_1+m)| $. Similarly, by the recurrence relation, there is a constant $M_2 \in \mathbb{R}$ and $a_1 \in \mathbb{C}$ with $1 < \Re(a_1) + 1 < \Re(c_1)$ such that $|\Gamma(a+m) | \le_K m^{M_2} |\Gamma(a_1+m)|$. Note that by \cite[5.6.6]{NIST}, $|\Gamma(a_1+m)| \le_K |\Gamma(\Re a_1+m)| $. Hence, by \cite[5.6.8]{NIST}, there is a constant  $M_3$ such that 
\begin{equation}
    \left|  \frac{\Gamma(a+m)}{\Gamma(c+m)} \right| \le_K   m^{M_1+M_2}  \left|  \frac{\Gamma(\Re a_1+m)}{\Gamma(\Re c_1+m)} \right| \le_K m^{M_3}.
\end{equation}
We evaluate $\left| \frac{\Gamma(b+m)}{\Gamma(m)} \right| $ in a similar way. As a consequence, there exists a constant $M$ such that for $|z| \le_K 1$, the function 
\begin{equation}
|\, _2F_1(a, b; c; z) \, | = \left|\sum_{m=0}^{\infty} \frac{(a)_m (b)_m}{(c)_m m!} z^m \right| \le_K \sum_{n=0}^{\infty} m^M z^m,    
\end{equation}
which is bounded. 
Thus, for fixed $n$ and large $|n_1|$, we have that $  \, _2 F_1 (d+1,d+r_1+1;k;\frac{n}{n_1})$ is uniformly bounded as a function of $n_1$ for $r_1, r_2, d$ belonging to compact subsets.
This implies 
\begin{equation}
    |\mathcal{Q}(n_1,n_2)| \le_K |n_1|^{-\Re(d)-1}.
\end{equation}  
For any $\varepsilon>0$, we have 
\begin{equation}
|    \sigma_{-r_1}(n_1) | \le |n_1|^{\varepsilon + \max ( - \Re(r_1), 0 )}
\end{equation}
and every term on the left side of \eqref{eq:first_main_thm_statement_sdfsdffff}  can be evaluated by 
\begin{equation}\label{eq:est_on_Q}
    |\mathcal{Q}(n_1,n_2)  \sigma_{-r_1}(n_1) \sigma_{-r_2}(n_2)| \le_K |n_1|^{-\Re(d)-1} |n_1|^{ \varepsilon + \max (- \Re (r_1), 0) } |n-n_1|^{ \varepsilon + \max (- \Re (r_2), 0) } .
\end{equation}
Thus the left side of \eqref{eq:first_main_thm_statement_sdfsdffff} is an absolutely convergent sum as long as 

\begin{equation}
    \Re (d)  + \min ( \Re (r_1), 0) + \min ( \Re (r_1), 0) > 0.
\end{equation}
For example, it would be sufficient to assume 
\begin{equation}\label{eq:domain_of_interesedfsdffff}
    \Re (d) > 0, \quad  \Re (r_1),  \Re (r_2) \ge 0
\end{equation}
which is exactly the second and the third condition in \eqref{eq:cond_1}.

Moreover, by \cite[15.2(ii)]{NIST}, the principal branch of $\frac{1}{\Gamma(c)} \cdot \, _2 F_1(a, b; c; z)$ is an entire function of $a$, $b$, and~$c$. The same is true for other branches excluding $z = 0$, $1$, and $\infty$. Thus, for fixed $n_1$ with $n_1 \neq 0$ and  $n_1 \neq n$, the function $\mathcal{Q}(n_1, n-n_1)$ is entire as a function of $r_1, r_2, d$. The only subtlety is that instead of $\, _2 F_1 (\cdot, \cdot , \cdot ; z)$, we consider $ \tfrac{1}{2} (\, _2 F_1 (\cdot, \cdot , \cdot ; z+i 0)+\, _2 F_1 (\cdot, \cdot , \cdot ; z-i0))$, which we treat as a sum of two hypergeometric functions -- one on the principal branch, and another one on the next branch.

We still need to prove that any point $(r_1, r_2, d) \in \mathbb{C}^3$ with components satisfying  \eqref{cond_2a} can be connected to the domain of the original convergence by a continuous path $ \{\widetilde  r_1(t),   \widetilde r_2(t), \widetilde  d(t) \}$ with $t \in [0,1]$ such that:
\begin{enumerate}[a)]
    \item for small $1-t>0$, the coordinates $ \widetilde  r_1(t),  \widetilde  r_2 (t),  \widetilde  d(t)$ satisfy \eqref{cond_2a}, and thus for $t$ near $1$, the coordinates satisfy the hypothesis of Theorem~\ref{thm:THEOREM-A-precursor},
    \item $ \widetilde  r_1(0)=r_1$, $ \widetilde r_2(0)=r_2$, and $ \widetilde d(0)=d$.
\end{enumerate}
It is not complicated to check that, e.g., the path 
\begin{equation}
\begin{split}
  \widetilde   r_1(t) =  r_1 \cdot  (1-t) + \delta\cdot t, \quad 
   \widetilde  r_2(t) = r_2 \cdot (1-t) + (k - 2 + \delta) \cdot t , \quad  
    \widetilde  d(t) = d \cdot (1-t) +  ( \tfrac{1}{4} - \delta) \cdot t
\end{split}    
\end{equation}
for some small $\delta>0$ satisfies the mentioned properties. Meromorphically continuating\footnote{We also note that the meromorphic continuation does not depend on a path we chose due to the condition \eqref{cond_2a}.} along the path $\{\widetilde  r_1(t),   \widetilde r_2(t), \widetilde  d(t) \}$ starting with $t=1$ and ending at $t=0$ concludes the proof. 
\section{Regularization of certain divergent sums}\label{sec:regularization_of_divergent_sums}
The goal of this section is to find regularization of certain divergent series.  For example, we would like to regularize series of the form
\begin{equation}
        \sum_{ \stackrel{n_1, n_2 \in \mathbb{Z} \setminus \{0, n\}  }{n_1+n_2=n}}  P(n_1, n_2) \sigma_{-r_1}(n_1) \sigma_{-r_2}(n_2),
\end{equation}
where $P$ is a certain polynomial in $n_1$ and $n_2$.  The main motivation for addressing such sums is to account for the parameters $r_1, r_2$ and $d$ which appear in the theory of modular graph functions (see, e.g., \cite{DGW, DKS2021_2}) but which are not accounted for in the statement of Theorem~\ref{thm:THEOREM-A}. 

To regularize a divergent series of the type
\begin{equation}\label{eq:sum_to_be_regularized}
        \sum_{ \stackrel{n_1, n_2 \in \mathbb{Z} \setminus \{0, n\}  }{n_1+n_2=n}} f(n_1, n_2),
\end{equation}
we  introduce a deformation parameter $\varepsilon \in \mathbb{C}$ and a family of functions $f(\varepsilon; n_1, n_2)$ such that 
\begin{enumerate}[(i)]
 \item $f(\varepsilon; n_1, n_2)$ are meromorphic as functions of~$\varepsilon \in \mathbb{C}$, 
\item 
$
  f(\varepsilon; n_1, n_2) \to  f(n_1, n_2)
$
as $\varepsilon \to 0$,
and 
\item for $\varepsilon \in \mathbb{C}$ satisfying $\Re(\varepsilon) > 1$,  the sum 
\begin{equation}\label{eq:sum_F_sssllkfkkkwppoooe}
F(\varepsilon) =        \sum_{ \stackrel{n_1, n_2 \in \mathbb{Z} \setminus \{0, n\}  }{n_1+n_2=n}}  f(\varepsilon; n_1, n_2)
\end{equation}
is absolutely and uniformly convergent.
\end{enumerate}
Then, if it is possible to find an analytic continuation of $F(\varepsilon)$ as a function of $\varepsilon$ to a neighborhood of 0, we say that $F(0)$ \textit{regularizes} the sum \eqref{eq:sum_to_be_regularized}. 

Our next goal is to define a function $f(\varepsilon; n_1, n_2)$. 
Let $r_1,r_2 > 0$, $d < 0$ such that $r_1 + d > 0$  and $2d+r_1+r_2+2 = k \in 2 \mathbb{N}_0 + 6$, and define 
\begin{equation}\label{eq:paths_of_r1_and_r2-take2}
\begin{split}
  \hat   r_1(\varepsilon) &= r_1+\varepsilon, \\
\hat     r_2(\varepsilon) &= r_2+\varepsilon, \\
\hat    d(\varepsilon)   &= d-\varepsilon.\\
\end{split}
\end{equation}
Note that $  2 d(\varepsilon) +   r_1(\varepsilon)+  r_2(\varepsilon)+2 =  2d+r_1+r_2+2 = k.$
Without loss of generality assume $\Re(r_1)\le \Re(r_2)$, then let  $\varepsilon_0 = -r_1 + \delta$ and choose $\delta > 0$ so that
\begin{equation}\label{eq:paths_of_r1_and_r2_1}\begin{split}
    \Re( \hat r_1(\varepsilon_0)) &= \Re(\delta) > 0, \\
    \Re( \hat r_2(\varepsilon_0)) &= \Re( r_2 -r_1 +  \delta) > 0, \\
    \Re(\hat d(\varepsilon_0))   &= \Re(d+r_1 - \delta) > 0.
\end{split}\end{equation}
Let the function $f(\varepsilon; n_1, n_2)$ be defined as 
\begin{equation}\label{eq:f_eps_llalallll}
    f(\varepsilon; n_1, n_2) = \Gamma(d(\varepsilon)+1) \cdot  \sigma_{- \hat r_1 (\varepsilon) }(n_1) \sigma_{- \hat r_2 (\varepsilon)}(n_2) \mathcal{Q}^{(\hat r_1(\varepsilon), \hat r_2(\varepsilon) )}_{\hat d(\varepsilon)} (n_1,n_2),
\end{equation}
for $\mathcal{Q}$ as in \eqref{eq:Q_convolution_form_main_theorem}. 
The condition \eqref{eq:paths_of_r1_and_r2_1} guarantees that for $\varepsilon$ in a neighborhood of $\varepsilon_0$, the sum  \eqref{eq:sum_F_sssllkfkkkwppoooe} recovers (up to a factor) an absolutely convergent series from Theorem~\ref{thm:THEOREM-A}. In order to study the regularization, we connect $\varepsilon_0$ and $0$ by a path in the complex plane and perform a meromorphic continuation along such a path. 

Before doing so, we explain the choice to multiply our identities by $\Gamma(d(\varepsilon)+1)$. To account for parameters appearing in the theory of modular graph functions, we need to consider $d \in - \mathbb{N}$ and odd $r_1$ and $r_2$. Recall that for $n \in \mathbb{N}_0$,  we have \cite[15.9.1]{NIST}
\begin{equation}\label{eq:2F1_vs_Jacobi_functions}
{ }_2 F_1(-n, \alpha+1+\beta+n ; \alpha+1 ; x)=\frac{n! \Gamma(\alpha+1)}{ \Gamma(\alpha+n+1)} P_n^{(\alpha, \beta)}(1-2 x),
\end{equation}
where $P_n^{(\alpha, \beta)}$ denotes the Jacobi polynomial, 
and thus for $d \in -\mathbb{N}$, we get
\begin{equation} \label{equation_2f1_regularized_asdasdllkkkk}
G_+(n_1,n_2) = G_-(n_1,n_2) = \, _2F _1 (d+1,d+r_1+1;k;\tfrac{n}{n_1}) = \frac{(-d-1)! \Gamma(k) }{ \Gamma(k-d-1) } P^{  (k-1, -r_2) }_{-d-1} (1-2 \tfrac{n}{n_1}).
\end{equation}
Note that if we were to substitute $d \in - \mathbb{N}$  into \eqref{eq:Q_convolution_form_main_theorem} directly, the first term ${\left(  G_+  - G_-  \right) \cdot  i^{k+1} \sin( \pi r_2 / 2)}$ would vanish by \eqref{equation_2f1_regularized_asdasdllkkkk}, as $G_+(n_1, n_2)=G_-(n_1, n_2)$. Additionally, the term involving $\cos(\pi r_2 / 2)$ and $\cos(\pi r_1 / 2)$ would vanish as well since both $r_1$ and $r_2$ are odd. 
Thus when  $d \in -\mathbb{N}$, it makes sense to rather consider the convolution sum multiplied by $\Gamma(d(\varepsilon)+1)$.

We will impose further assumptions on the values of $d$ to avoid the poles of $L$-functions, as further seen in Section \ref{sec:regL-val}. In what follows, we consider how different terms in 
Theorem~\ref{thm:THEOREM-A} behave under the limit $\varepsilon \to 0$. 

\begin{remark}
Interestingly, the approach above resembles the regularization techniques used in the evaluation of Mellin–Barnes-type integrals, particularly in the context of Feynman diagram calculations (cf. \cite{smirnov2009resolution}):
\[
\frac{1}{(2\pi i)^n} \int_{-i\infty}^{+i\infty} \cdots \int_{-i\infty}^{+i\infty} 
\frac{ \prod_i \Gamma\left( a_i + b_i \varepsilon + \sum_j c_{ij} z_j \right) }
     { \prod_i \Gamma\left( a'_i + b'_i \varepsilon + \sum_j c'_{ij} z_j \right) }
\prod_k x_k^{d_k} \prod_{l=1}^n \mathrm{d}z_l,
\]
where the regularization of a single integral is obtained via the analytic continuation in the  deformation parameter, cf. Lemma \ref{lemma_gamma_function_stuflöskdflkdjlkfjlksdfjlsd}. In our case, however, the divergence manifests on the level of the whole sum.
\end{remark}

\subsection{Contribution from $Z^{\alpha, \beta}_d$}
In this section, we study how $ \Gamma(1+\hat d(\varepsilon)) Z^{(\hat r_1(\varepsilon),\hat r_2(\varepsilon))}_{\hat d(\varepsilon)}$ behaves under the degeneration. 
\begin{lemma} For $r_1,r_2>1$ odd integers and $d<0 $ an integer, 
\begin{align}
\lim_{\varepsilon\to 0}&\left(\Gamma(1+\hat d(\varepsilon)) \cdot Z^{(\hat r_1(\varepsilon),\hat r_2(\varepsilon))}_{\hat d(\varepsilon)}\right)\nonumber\\ &= 
  \frac{(-1)^{d+1} (2 \pi )^{r_2} }{\Gamma (1+d+r_2) \Gamma (1+d+r_1+r_2 ) \Gamma(-d) } \zeta'(1-r_2) + \frac{(2 \pi )^{-r_2} \zeta (1+r_2 )}{\Gamma (1+d+r_1 )} n^{-r_2}.
\end{align}
\end{lemma}
\begin{proof}
Replacing \(\alpha, \beta, d\) by \(\hat r_1(\varepsilon), \hat r_2(\varepsilon),\hat d(\varepsilon)\) respectively in $Z^{(\alpha,\beta)}_d$ as defined in \eqref{def:Zabd}, we find that its first term multiplied by \(\Gamma(\hat d(1+\varepsilon))\) becomes 
\begin{equation}
\frac{(2 \pi )^{r_2+\varepsilon} }{\Gamma (1+d+r_2) \Gamma (1+d+r_1+r_2+\varepsilon )} \zeta (1-r_2-\varepsilon ) \Gamma \left(1+d -\varepsilon\right).
\end{equation}
Note that 
\begin{equation}
  \frac{(2 \pi )^{r_2+\varepsilon} }{\Gamma (1+d+r_2) \Gamma (1+d+r_1+r_2+\varepsilon )} 
  \stackrel{\varepsilon \to 0}{\to}
  \frac{(2 \pi )^{r_2} }{\Gamma (1+d+r_2) \Gamma (1+d+r_1+r_2 )} ,
\end{equation}
and the value on the right is finite since the reciprocal of the gamma function is entire. 
For an odd integer $r_2>1$, the function $   \zeta (1-r_2-\varepsilon)$ has a simple zero at $\varepsilon = 0$, and $\Gamma \left(1+d -\varepsilon\right)$ has a simple pole. Recall that for $n\in\mathbb{Z}_{\leq 0}$, one has  $\text{Res}_{s=n}\Gamma(s) = (-1)^n/(-n)! = (-1)^n / \Gamma(1-n)$. Hence, for $r_1,r_2>1$ odd integers and $d<0 $ an integer,
\begin{equation}
   \zeta (1-r_2-\varepsilon ) \Gamma \left(1+d -\varepsilon\right)
    \stackrel{ \varepsilon \to 0  }{\to} 
    \frac{  (-1)^{d+1}  }{    \Gamma(-d)}    \zeta'(1-r_2).
\end{equation}
On the other hand, the second term in \eqref{def:Zabd} multiplied by $\Gamma(d(\varepsilon)+1)$ becomes 
\begin{equation}
\frac{(2 \pi )^{-r_2-\varepsilon } \zeta (1+r_2+\varepsilon )}{\Gamma (1+d+r_1+\varepsilon )}  n^{-r_2-\varepsilon} \stackrel{  \varepsilon \to 0   }{\to} \frac{(2 \pi )^{-r_2} \zeta (1+r_2 )}{\Gamma (1+d+r_1 )} n^{-r_2}.
\end{equation}
\end{proof}

\subsection{Contribution from $L$-values}\label{sec:regL-val}
We consider the contribution from $L$-values to the convolution identities \eqref{eq:first_main_thm_statement_sdfsdffff} and multiply by $\Gamma(\hat d(\varepsilon)+1)$ with \eqref{eq:paths_of_r1_and_r2-take2} substituted. We can further let  $\varepsilon \to 0$  
and obtain that the meromorphic continuation of $L$-function contribution is equal to:
\begin{equation}
-\frac{i^k 2^{2 - 2 k - r_2} \pi ^{1-k-r_2} \Gamma (k-1) }{\Gamma \left(1+d+r_1\right)} n^{-d-r_1-r_2} \sum_{f \in \mathcal{F}_k} \frac{L\left(1 + d + r_2,f\right) L\left(1 + d + r_1 + r_2,f\right)}{\langle \bar f, \bar f \rangle }   a_f(n).
\end{equation}

  The poles of $L(s, f)$ are at $s=0$ and $s=k$. The former corresponds to $d=-r_2-1$ or $d=-1-r_1-r_2$, and the latter corresponds to $d=-1$ or $d=-r_1-1$.  In the cases considered in \cite{DKS2021_2} these parameters are avoided.

\subsection{The behavior of $\mathcal{Q}$}
To slightly simplify the notations in \eqref{eq:f_eps_llalallll}, denote 
\begin{equation}
G_{\pm}(\varepsilon) = G_{\pm}(n_1, n_2; \varepsilon) = \, _2 F_1 \left(\hat d(\varepsilon)+1,\hat d(\varepsilon)+\hat r_1(\varepsilon)+1;k;\frac{n}{n_1} \pm  i 0\right),
\end{equation} 
and define  $Q(n_1, n_2; \varepsilon)$  by (cf. \eqref{eq:Q_convolution_form_main_theorem})
    \begin{equation}\label{eq:Q_convolution_form-take2}
    \begin{split}
    \Gamma(k) \left|\frac{n_1}{n}\right|^{\hat d(\varepsilon)+1}  \mathcal{Q}(n_1,n_2; \varepsilon) = & \left(  G_+(\varepsilon)  - G_-(\varepsilon)  \right) \cdot  i^{k+1} \sin( \pi \hat r_2(\varepsilon) / 2) \\
    + & \left(  G_+(\varepsilon)  + G_-(\varepsilon) \right) \cdot \begin{cases}
       i^{k}  \cos(\pi \hat r_2(\varepsilon) / 2), \quad &n_1 > 0, \\
             \cos(\pi \hat r_1(\varepsilon) / 2), \quad &n_1 < 0. 
        \end{cases}
    \end{split}
\end{equation}

The following lemma describes the behavior of the convolution weight (multiplied by $\Gamma(\hat d(\varepsilon) +1   ) $) as $\varepsilon \to 0$: 
\begin{lemma} Let $n \in \mathbb{N}$ and let  $n_1,n_2\in\mathbb{Z}$ so that $n_1+n_2=n$. Let $r_1,r_2>1$ be odd integers, and let $d<0 $ an integer with $\Re(d+r_1) > 0$ and $\Re(d+r_2) > 0$, then 
\begin{equation}\label{eq:limitQ}
    \begin{split}
&  \Gamma(k-d-1) \cdot \lim_{\varepsilon \to 0} \left(  \Gamma(\hat d(\varepsilon) +1   )      \mathcal{Q}(n_1,n_2; \varepsilon)  \right) \\
& \ \ \ \ \ \ \ \ \ \ \ \ \ \ \ \ \ \ \ \ \ \ =    2  \pi i^{k+r_2+1}   \left( \frac{n_2}{n}  \right)^{r_2}  P_{d+r_1}^{(r_2,-r_1)}\left(2\tfrac{ n_1}{n}-1\right) \cdot \begin{cases}
    1  , & 0 < n_1, n_2 < n, \\
    0, & \text{otherwise},
\end{cases}
\\
& \ \ \ \ \ \ \ \ \ \ \  \ \ \ \ \ \ \ \ \ \ \ +      \pi i^{ r_1+1 }   \left(\frac{n}{n_1}\right)^{ d+1}   P^{  (k-1, -r_2) }_{-d-1} (1-2 \tfrac{n}{n_1}) \cdot \begin{cases}
    1  , & n_1 > 0, \\
    -1, & n_1 < 0.
\end{cases}
    \end{split}
\end{equation}

\end{lemma}

\begin{proof}
The proof will follow from multiplying the right hand side of \eqref{eq:Q_convolution_form-take2} by $\Gamma(\hat d(\varepsilon) +1   )$ and letting $\varepsilon \to 0$. 
We start with its first summand. Note that 
\begin{equation}\label{eq:first_term_in_convolution_reg111223123}
    i^{k+1} \sin( \pi \hat r_2(\varepsilon) / 2)  \stackrel{  \varepsilon \to 0  }{\to } i^{k+1} \sin( \pi r_2 / 2).
\end{equation}
From Remark \ref{rem:remark_on_difference_between_cuts}, 
    \begin{equation}\label{eq:limit_finite_part}
    \begin{split}
     \Gamma(\hat d(\varepsilon)+1 )   &\left(   G_+(\varepsilon)  - G_-(\varepsilon) \right)   
     =  \delta_{ n/n_1 > 1 } \cdot 2  \pi i   \frac{  \Gamma(2 \hat d(\varepsilon)+\hat r_1(\varepsilon)+\hat r_2(\varepsilon)+2)  }{\Gamma(\hat r_2(\varepsilon)+1)\Gamma (\hat d(\varepsilon)+\hat r_1(\varepsilon)+1)}   (\tfrac{n_2}{n_1})^{\hat r_2 (\varepsilon)} \\
     &  \ \ \ \ \ \ \ \ \ \ \ \ \ \ \ \ \ \ \ \ \ \ \ \ \ \ \ \times\, _2F_1( \hat d(\varepsilon)+\hat r_1(\varepsilon)+\hat r_2(\varepsilon)+1,\hat d(\varepsilon)+\hat r_2(\varepsilon)+1;\hat r_2(\varepsilon)+1;1-\tfrac{n}{n_1}) \\ 
  & \stackrel{\varepsilon \to 0}{\to}  \delta_{ n/n_1 > 1 } \cdot  2  \pi i   \frac{  \Gamma(k) }{\Gamma(r_2+1)\Gamma (d+r_1+1)}  (\tfrac{n_2}{n_1})^{r_2} \, _2F_1(k-d-1,d+r_2+1;r_2+1;1-\tfrac{n}{n_1}) .
    \end{split}
    \end{equation}
From \cite[15.8.1]{NIST}, when $-d-r_1\in\mathbb{Z}_{<0}$,
\begin{equation}\label{eq:2f1_sdfjjfjfjjjfdjd_right_kjskskskdkd_iiiifkkkk}
    \begin{split}
     \, _2F_1(k-d-1,d+r_2+1;r_2+1;1-\tfrac{n}{n_1})&  =  \left(\frac{n}{n_1}\right)^{-d-r_2-1} \, _2 F_1(  -d-r_1, d+r_2+1, r_2+1, 1-\tfrac{n_1}{n}) \\
&   \stackrel{\eqref{eq:2F1_vs_Jacobi_functions}}{=}  \left(\frac{n}{n_1}\right)^{-d-r_2-1}  \frac{(d+r_1)! }{(r_2+1)_{d+r_1}} P_{d+r_1}^{(r_2,-r_1)}\left(2\tfrac{n_1}{n}-1\right).
\end{split}
\end{equation}
Combining together \eqref{eq:first_term_in_convolution_reg111223123}, \eqref{eq:limit_finite_part} and \eqref{eq:2f1_sdfjjfjfjjjfdjd_right_kjskskskdkd_iiiifkkkk}, we obtain the first summand in the right hand side of \eqref{eq:limitQ}.

Now consider the second summand in the right hand side of \eqref{eq:Q_convolution_form-take2}. The definition of $\hat r_1(\varepsilon), \hat r_2(\varepsilon)$ and $\hat d(\varepsilon)$ implies  
\begin{equation}\label{eq:gammacosstuff}
    \Gamma (\hat d(\varepsilon)+1) \cos \left(\frac{\pi  \hat r_1(\varepsilon) }{2}\right) = \cos \left(\frac{r_1+\varepsilon}{2} \pi  \right) \Gamma \left(d-\varepsilon +1\right).
\end{equation}
For $r_1,r_2>1$ odd integers and $d<0 $ an integer, we have 
\begin{equation}
    \cos \left(\frac{r_1+\varepsilon}{2} \pi  \right) \Gamma \left(d-\varepsilon +1\right) \stackrel{ \varepsilon \to 0 }{\to}  \frac{ (-1)^{\frac{r_1+1}{2}+d} \pi   }{2\Gamma(-d)},
\end{equation}
hence 
\begin{equation}\label{eq:gammacosstuffdf}
\begin{split}
    i^k \Gamma (\hat d(\varepsilon)+1) \cos \left(\frac{\pi  \hat r_2(\varepsilon) }{2}\right) & \stackrel{ \varepsilon \to 0 }{\to}   \frac{ (-1)^{\frac{k+r_2+1}{2}+d} \pi   }{2\Gamma(-d)},\\
    \Gamma (\hat d(\varepsilon)+1) \cos \left(\frac{\pi  \hat r_1(\varepsilon) }{2}\right) & \stackrel{ \varepsilon \to 0 }{\to}   \frac{ (-1)^{\frac{r_1+1}{2}+d} \pi   }{2\Gamma(-d)}.
\end{split}
\end{equation}
Moreover, by \eqref{equation_2f1_regularized_asdasdllkkkk}, we have 
\begin{equation}\label{G+andG-combinedtogetherwithepsilon}
    G_+(\varepsilon) + G_- (\varepsilon)\stackrel{\varepsilon \to 0}{\to} 2 \frac{\Gamma(-d) \Gamma(k)}{ \Gamma(k-d-1) } P^{  (k-1, -r_2) }_{-d-1} (1-2 \tfrac{n}{n_1}).
\end{equation}
Combining the first line in \eqref{eq:gammacosstuffdf} with  \eqref{G+andG-combinedtogetherwithepsilon} and using the parity of $r_1$, $r_2$ together with $k=2d+2+r_1+r_2$, we obtain the second line of \eqref{eq:limitQ} restricted to $n_1>0$. Similarly, combining the second line in \eqref{eq:gammacosstuffdf} with  \eqref{G+andG-combinedtogetherwithepsilon}, we obtain the second line of \eqref{eq:limitQ} restricted to $n_1<0$.

\end{proof}

 \section{Appendix: Special cases of $\mathcal{Q}(n_1,n_2)$}\label{sec:appendix}
  In the appendix, we identify some special cases where $\mathcal{Q}(n_1,n_2)$ (or, more precisely, $\, _2 F_1 (d+1,d+r_1+1;k;\cdot )$ ) simplifies to an elementary function. 
  Throughout, we assume that $d, r_1$ and $r_2$
 satisfy the conditions in \eqref{eq:cond_1}.
  
  We start by noting that the hypergeometric function $\, _2 F_1 (d+1,d+r_1+1;k;\cdot )$ is never a function from Schwarz's list \cite{schwarz1873ueber}, because its third parameter, $k$,  is an integer. 
\subsection{$d$ is a positive integer, $r_1, r_2 \in \mathbb{C}$}\label{eq:sec_sdddfkkeee}

For any $a \in \mathbb{C}$, we have  
\begin{equation}
    \, _2F_1(1,a;1;z) = (1-z)^{-a}.
\end{equation}
Together with formulas that allow to increase the first and the third parameter in hypergeometric series by natural numbers \cite[15.5.3, 15.5.6]{NIST},
 it implies that  
\begin{equation}
\, _2 F_1 (d+1,a;k; z)
\end{equation}
is an elementary function which can be expressed in terms of rational functions and logarithms. From the definition of $G_\pm$, the function $\mathcal{Q}(n_1, n_2)$ is an elementary function when $d$ is a positive integer. It does not matter whether $r_1$ and $r_2$ are integers as long as $2d+2+r_1+r_2$ is a positive integer and $d$ is an integer.
In the case when $r_1$ is an integer, it would be most convenient to take 
\begin{equation}
\, _2F_1(1,1;2;z) = -z^{-1} \log(1 - z).
\end{equation}

\subsection{$d + r_1$ is an integer (and hence, $d + r_2$ is an integer)}
If $d+r_1$ is an integer, by the symmetry of the hypergeometric function with respect to changing the first two arguments and rising the arguments as in Section \ref{eq:sec_sdddfkkeee}, 
$
    \, _2 F_1 (\cdot , d+r_1+1 ;k; z)
$
is an elementary function. If $d+r_2$ is an integer, then $d+r_1 = (2d+r_1+r_2+2)  - (d+r_2) - 2$ is also an integer, and we proceed as before.

\subsection{$d$ is a positive half-integer, $r_1$ integer, $r_2$ integer }
From \begin{equation}
    \, _2F_1\left(\frac{1}{2},\frac{1}{2};1;z\right) = \frac{2 K(z)}{\pi },
\end{equation}
where $K(z)$ is a an elliptic integral, 
and rising arguments as above, we obtain that for $d$ is half-integer, $r_1$ integer, $r_2$ integer, the convolution weight is expressed with the help of elliptic integrals.

\section*{Acknowledgements}
We would like to thank Nikolaos Diamantis, Daniele Dorigoni, Michael Green, Axel Kleinschmidt, Stephen D. Miller, Danylo Radchenko, Oliver Schlotterer, and Don Zagier for their many helpful conversations and insights. We would like to thank Axel Sorge for the help with calculations.   K.K.-L. acknowledges support from the NSF through grant DMS-2302309.  K.F. is partially funded by the Deutsche Forschungsgemeinschaft (DFG, German Research Foundation) under Germany's Excellence Strategy EXC 2044-390685587, Mathematics M\"unster: Dynamics-Geometry-Structure.

\bibliography{bib2}

@article {CGPWW2021,
    AUTHOR = {Chester, S. M. and Green, M. B. and Pufu, S. S. and
              Wang, Y. and Wen, C.},
     TITLE = {New modular invariants in {$\mathcal{N}=4$} super-{Y}ang-{M}ills
              theory},
   JOURNAL = {J. High Energy Phys.},
  FJOURNAL = {Journal of High Energy Physics},
      YEAR = {2021},
    NUMBER = {4},
     PAGES = {Paper No. 212, 56},
      ISSN = {1126-6708},
   MRCLASS = {81T60 (81T35 81T40)},
       DOI = {10.1007/jhep04(2021)212},
       URL = {https://doi.org/10.1007/jhep04(2021)212},
}

@article {DKS2021_1,
    AUTHOR = {Dorigoni, D. and Kleinschmidt, A. and Schlotterer,
              O.},
     TITLE = {Poincar\'{e} series for modular graph forms at depth two. {P}art
              {I}. {S}eeds and {L}aplace systems},
   JOURNAL = {J. High Energy Phys.},
  FJOURNAL = {Journal of High Energy Physics},
      YEAR = {2022},
    NUMBER = {1},
     PAGES = {Paper No. 133, 111},
      ISSN = {1126-6708},
   MRCLASS = {81T30 (11F11)},
       DOI = {10.1007/jhep01(2022)133},
       URL = {https://doi.org/10.1007/jhep01(2022)133},
}

@article {DKS2021_2,
    AUTHOR = {Dorigoni, D. and Kleinschmidt, A. and Schlotterer,
              O.},
     TITLE = {Poincar\'{e} series for modular graph forms at depth two. {P}art
              {II}. {I}terated integrals of cusp forms},
   JOURNAL = {J. High Energy Phys.},
  FJOURNAL = {Journal of High Energy Physics},
      YEAR = {2022},
    NUMBER = {1},
     PAGES = {Paper No. 134, 49},
      ISSN = {1126-6708},
   MRCLASS = {81T30 (11F11)},
       DOI = {10.1007/jhep01(2022)134},
       URL = {https://doi.org/10.1007/jhep01(2022)134},
}

@article{DGW,
  title={Modular Features of Superstring Scattering Amplitudes: Generalised {E}isenstein Series and Theta Lifts},
  author={Dorigoni, D. and Green, M.B. and Wen, C.},
  journal={Ann. Henri Poincar\'{e} },
  year={2025},
  publisher={Springer},
    DOI = {10.1007/s00023-025-01607-6}
}

@misc{NIST,
         key = "{\relax DLMF}",
       title = "{\it NIST Digital Library of Mathematical Functions}",
howpublished = "http://dlmf.nist.gov/, Release 1.1.6 of 2022-06-30",
         url = "http://dlmf.nist.gov/",
        note = "F.~W.~J. Olver, A.~B. {Olde Daalhuis}, D.~W. Lozier, B.~I. Schneider,
                R.~F. Boisvert, C.~W. Clark, B.~R. Miller, B.~V. Saunders,
                H.~S. Cohl, and M.~A. McClain, eds."}

@article{SDK,
  title={{The {$D^6R^4$} interaction as Poincar\'{e} series, and a related shifted convolution sum}},
  author={Klinger-Logan, K. and Miller, S. and Radchenko, D.},
  journal={preprint},
  year={2022}
}

@book{motohashi1997spectral,
  title={Spectral theory of the {R}iemann zeta-function},
  author={Motohashi, Y.},
  volume={127},
  year={1997},
  publisher={Cambridge University Press}
}

@article{murty2013ramanujan,
  title={Ramanujan series for arithmetical functions},
  author={Murty, M. R.},
  journal={Hardy-Ramanujan Journal},
  volume={36},
  year={2013},
  publisher={Episciences. org}
}

@article{ramanujan1918certain,
  title={On certain trigonometrical sums and their applications in the theory of numbers},
  author={Ramanujan, S.},
  journal={Trans. Cambridge Philos. Soc},
  volume={22},
  number={13},
  pages={259--276},
  year={1918}
}

@book{iwaniec2021spectral,
  title={Spectral methods of automorphic forms},
  author={Iwaniec, H.},
  volume={53},
  year={2021},
  publisher={American Mathematical Society, Revista Matem{\'a}tica Iberoamericana }
}

@article{kiuchi1987exponential,
  title={On an exponential sum involving the arithmetic function $\sigma_a(n)$},
  author={Kiuchi, I.},
  journal={Mathematical Journal of Okayama University},
  volume={29},
  number={1},
  pages={193--205},
  year={1987},
  publisher={Department of Mathematics, Faculty of Science, Okayama University}
}

@article{ishibashi1995,
  title={The value of the {E}stermann zeta functions at $s=0$},
  author={Ishibashi, M.},
  journal={Acta Arithmetica},
  volume={73},
  number={4},
  pages={357--361},
  year={1995}
}

@book{apostol,
  title={Introduction to analytic number theory},
  author={Apostol, T. M.},
  year={1998},
  publisher={Springer Science \& Business Media}
}

@inproceedings{motohashi1994binary,
  title={The binary additive divisor problem},
  author={Motohashi, Y.},
  booktitle={{Annales scientifiques de l'Ecole normale sup{\'e}rieure}},
  volume={27},
  number={5},
  pages={529--572},
  year={1994}
}

@article{diamantis2010kernels,
  title={Kernels of {L}-functions of cusp forms},
  author={Diamantis, N. and O’Sullivan, C.},
  journal={Mathematische Annalen},
  volume={346},
  pages={897--929},
  year={2010},
  publisher={Springer}
}

@article{kuznetsov1985convolution,
  title={{Convolution of the Fourier coefficients of the Eisenstein-Maass series}},
  author={Kuznetsov, N.V.},
  journal={Journal of Soviet Mathematics},
  volume={29},
  pages={1131--1159},
  year={1985},
  publisher={Springer}
}

@book{kontsevich2001periods,
  title={Periods},
  author={Kontsevich, M. and Zagier, D.},
  year={2001},
  publisher={Springer}
}

@article{schwarz1873ueber,
  title={{Ueber diejenigen F{\"a}lle, in welchen die Gaussische hypergeometrische Reihe eine algebraische Function ihres vierten Elementes darstellt.}},
  author={Schwarz, H. A.},
  year={1873},
  publisher={De Gruyter}
}

@article{smirnov2009resolution,
  title={{On the resolution of singularities of multiple Mellin--Barnes integrals}},
  author={Smirnov, A. V. and Smirnov, V. A.},
  journal={The European Physical Journal C},
  volume={62},
  number={2},
  pages={445--449},
  year={2009},
  publisher={Springer}
}

@book{prud,
  title={Integrals and series},
  author={Prudnikov, A.B.},
  year={2018},
  publisher={Routledge}
}

@incollection{o2023identities,
  title={Identities from the holomorphic projection of modular forms},
  author={O’Sullivan, C.},
  booktitle={Number Theory for the Millennium III},
  pages={87--106},
  year={2023},
  publisher={AK Peters/CRC Press}
}

@article{FKLR,
  title={Convolution identities for divisor sums and modular forms},
  author={Fedosova, K. and Klinger-Logan, K. and Radchenko, D.},
  journal={Proceedings of the National Academy of Sciences},
  volume={121},
  number={44},
  pages={e2322320121},
  year={2024},
  publisher={National Academy of Sciences}
}

@article{fedosova2024shifted,
  title={Shifted convolution sums motivated by string theory},
  author={Fedosova, K. and Klinger-Logan, K.},
  journal={Journal of Number Theory},
  volume={260},
  pages={151--172},
  year={2024},
  publisher={Elsevier}
}

@article{diamantis2020kernels,
  title={Kernels of {$L$}-functions and shifted convolutions},
  author={Diamantis, N.},
  journal={Proceedings of the American Mathematical Society},
  volume={148},
  number={12},
  pages={5059--5070},
  year={2020}
}
\bibliographystyle{amsplain}
 
\end{document}